\documentclass[12pt]{amsart}
\usepackage{amsmath, amssymb, amsfonts}
\usepackage[pdftex]{graphicx}
\usepackage{verbatim} 
\usepackage{tikz}
\usepackage{tikz-cd} 
\usepackage{enumitem}
\usepackage{amssymb,tikz}
\usetikzlibrary{positioning}
  \def\gn#1#2{{$\href{http://groupnames.org/\#?#1}{#2}$}}
\def\gn#1#2{$#2$}  
\tikzset{sgplattice/.style={inner sep=1pt,norm/.style={black!50!black},char/.style={black!50!black},
  lin/.style={black!50}},cnj/.style={black!50,yshift=-2.5pt,left=-1pt of #1,scale=0.5,fill=white}}

\theoremstyle{plain}
\newtheorem{theorem}{Theorem}[section]
\newtheorem*{theorem*}{Theorem}
\newtheorem*{prop*}{Proposition}
\newtheorem{lemma}[theorem]{Lemma}
\newtheorem{corollary}[theorem]{Corollary}
\newtheorem{prop}[theorem]{Proposition}

\theoremstyle{remark}

\newtheorem{remark}[theorem]{Remark}

\newtheorem{example}[theorem]{Example}

\newtheorem*{note*}{Note}
\newtheorem*{remark*}{Remark}
\newtheorem*{example*}{Example}

\theoremstyle{definition}
\newtheorem*{definition*}{Definition}
\newtheorem*{hypothesis*}{Hypothesis}
\newtheorem*{assumptions*}{Assumptions}
\newtheorem{definition}[theorem]{Definition}

\newcommand{\Z}{\mathbb{Z}}

\newcommand{\Q}{\mathbb{Q}}

\newcommand{\Gal}{\mathrm{Gal}}
\newcommand{\Tr}{\mathrm{Tr}}

\newcommand{\ncl}{\mathrm{ncl}}

\numberwithin{equation}{section}

\newcommand{\Ind}{\mathrm{Ind}}

\title[Leopoldt-type theorems for non-abelian extensions of $\Q$]{Leopoldt-type theorems for \\ non-abelian extensions of $\Q$}

\author{Fabio Ferri}
\address[]{\parbox{\linewidth}{Department of Mathematics, University of Exeter, Rennes Dr, Exeter EX4 4RN, United Kingdom\vspace{0.15cm}}}
\email{ff263@exeter.ac.uk}
\urladdr{http://emps.exeter.ac.uk/mathematics/staff/ff263}

\date{version of 8th April 2022}
\begin{document}

\maketitle

\begin{abstract}
We prove new results concerning the additive Galois module structure of certain wildly ramified finite non-abelian extensions of $\Q$. In particular, when $K/\Q$ is a Galois extension with Galois group $G$ isomorphic to $A_4$, $S_4$ or $A_5$, we give necessary and sufficient conditions for the ring of integers $\mathcal{O}_{K}$ to be free over its associated order 
in the rational group algebra $\Q[G]$.
\end{abstract}

\section{Introduction}

Let $K/F$ be a finite Galois extension of number fields or $p$-adic fields and let $G=\Gal(K/F)$.
The classical normal basis theorem says that $K$ is free of rank $1$ as a module over the group algebra $F[G]$. 
A much more difficult problem is that of determining whether the ring of integers $\mathcal{O}_{K}$ is free of rank $1$ over an appropriate $\mathcal{O}_{F}$-order in $F[G]$. 
The natural choice of such an order is the so-called associated order
\[
\mathfrak{A}_{K/F}:=\{ \lambda \in F[G]:  \lambda \mathcal{O}_K \subseteq \mathcal{O}_K \},
\]
since this is the only $\mathcal{O}_{F}$-order in $F[G]$ over which $\mathcal{O}_{K}$ can possibly be free. 

It is clear that the group ring $\mathcal{O}_{F}[G]$ is contained in $\mathfrak{A}_{K/F}$.
In fact, $\mathfrak{A}_{K/F}=\mathcal{O}_{F}[G]$ if and only if $K/F$ is at most tamely ramified.
It is in this setting that by far the most progress has been made and we say that 
$K/F$ has a normal integral basis if  $\mathcal{O}_{K}$ is free over $\mathcal{O}_{F}[G]$.
The celebrated Hilbert-Speiser theorem says that if $K/\Q$ is a tamely ramified finite abelian extension,
then it has a normal integral basis. 
Leopoldt removed the assumption on ramification to obtain the following generalisation of this result.

\begin{theorem}\cite{MR0108479}\label{thm:leopoldt}
Let $K/\Q$ be a finite abelian extension. Then $\mathcal{O}_{K}$ is free over $\mathfrak{A}_{K/\Q}$.
\end{theorem}

Leopoldt also specified a generator and the associated order; Lettl \cite{MR1037435} gave a simplified and more explicit proof of the same result. We also have the following result of Berg\'e.

\begin{theorem}\cite{MR0371857}\label{thm:berge}
Let $p$ be a prime and let $K/\Q$ be a dihedral extension of degree $2p$. 
Then $\mathcal{O}_{K}$ is free over $\mathfrak{A}_{K/\Q}$.
\end{theorem}

Now let $K/\Q$ be a Galois extension with $\Gal(K/\Q)\cong Q_{8}$, the quaternion group of order $8$.
Suppose that $K/\Q$ is tamely ramified.
Martinet \cite{MR0291208} gave two examples of such extensions without and one with a normal integral basis. Moreover, Fr\"ohlich \cite{MR323759} showed that both possibilities occur infinitely often.
By contrast, in the case that $K/\Q$ is wildly ramified, we have the following result of Martinet.

\begin{theorem}\cite{MR299593}\label{thm:martinet}
Let $K/\Q$ be a wildly ramified Galois extension with $\Gal(K/\Q)\cong Q_{8}$. 
Then $\mathcal{O}_{K}$ is free over $\mathfrak{A}_{K/\Q}$.
\end{theorem}
 
In the present article, we prove other Leopoldt-type theorems for certain non-abelian extensions of $\Q$. 
An important notion is that of local freeness, which we now review.

For the rest of the introduction, let $K/\Q$ be a finite Galois extension and let $G=\Gal(K/\Q)$.
We recall that $\mathcal{O}_{K}$ is said to be locally free over $\mathfrak{A}_{K/\Q}$ at a rational prime $p$ 
if $\mathcal{O}_{K,p}:=\Z_p\otimes_\Z\mathcal{O}_K$ is free as an
$\mathfrak{A}_{K/\Q,p}:=\Z_p\otimes_\Z\mathfrak{A}_{K/\Q}$-module. 
We say that $\mathcal{O}_{K}$ is locally free over $\mathfrak{A}_{K/\Q}$
if this holds for all rational primes $p$.

Suppose that $K/\Q$ is tamely ramified. 
Then $\mathcal{O}_{K}$ is locally free over $\Z[G]$.
Moreover, the problem of determining whether $K/\Q$ has a normal integral basis is well understood thanks to Taylor's proof of Fr\"ohlich's conjecture \cite{MR608528}: he determined the class of $\mathcal{O}_{K}$ in the locally free class group $\text{Cl}(\Z[G])$ in terms of Artin root numbers of the irreducible symplectic characters of $G$ (see \cite[I]{MR717033} for an overview).
In particular, if $G$ has no irreducible symplectic characters
(this is the case, for instance, if $G$ is abelian, dihedral or of odd order), 
then $K/\Q$ has a normal integral basis. 
 
If $K/\Q$ is wildly ramified, then the situation becomes much more difficult, not least because 
$\mathcal{O}_{K}$ need not even be locally free over $\mathfrak{A}_{K/\Q}$. 
For instance, Berg\'e \cite{zbMATH03623711} gave examples of wildly ramified dihedral extensions of $\Q$
without the local freeness property. 

Let $N/M$ be a finite Galois extension of $p$-adic fields.
One can consider the analogous problem of whether $\mathcal{O}_{N}$ is free over 
$\mathfrak{A}_{N/M}$. Indeed, this is the case when $N/M$ is unramified, tamely ramified or weakly ramified,
or $M=\Q_{p}$ and $N/\Q_{p}$ is abelian or dihedral of order $2\ell$ for some prime $\ell$
(see \S \ref{padicresults} for a detailed overview of such results).
However, freeness in this situation does not relate to the aforementioned notion of local freeness in the way one might expect.
More precisely, if $K/\Q$ is wildly ramified and non-abelian,
$p$ is a rational prime, $\mathfrak{P}$ a prime of $K$ above $p$, 
and $\mathcal{O}_{K_\mathfrak{P}}$ is free over $\mathfrak{A}_{K_\mathfrak{P}/\Q_{p}}$, 
then it is not necessarily the case that $\mathcal{O}_{K}$ is locally free over $\mathfrak{A}_{K/\Q}$ at $p$
(here $K_{\mathfrak{P}}$ denotes the completion of $K$ at $\mathfrak{P}$).
This is an important obstacle that needs to be overcome in the proofs of the main results of the present article.

We consider the question of whether $\mathcal{O}_{K}$ is free over $\mathfrak{A}_{K/\Q}$
in the case that the locally free class group $\text{Cl}(\Z[G])$ is trivial and $\Z[G]$ has the 
so-called locally free cancellation property (in fact, the latter condition holds whenever the former holds).
In this situation, it is also the case that $\text{Cl}(\mathfrak{A}_{K/\Q})$ is trivial and
$\mathfrak{A}_{K/\Q}$ has the locally free cancellation property. 
Hence it is straightforward to show that the question reduces to whether 
$\mathcal{O}_{K}$ is locally free over  $\mathfrak{A}_{K/\Q}$.

We now briefly describe a straightforward application of this strategy. 
Suppose that $G$ is dihedral of order $2p^{n}$ for some prime $p$ and some positive integer $n$.
In this situation, Keating \cite{MR0357591} gave sufficient conditions for $\text{Cl}(\Z[G])$ to be trivial
and Berg\'e \cite{zbMATH03623711} gave necessary and sufficient conditions
for $\mathcal{O}_{K}$ to be locally free over $\mathfrak{A}_{K/\Q}$.
As a consequence we obtain the following result.

\begin{theorem}\label{thm:dihedralintro}
Let $n$ be a positive integer and let $p\geq 5$ be a regular prime number such that the class number of 
$\Q(\zeta_{p^{n}})^+$ is $1$. 
Let $K/\Q$ be a dihedral extension of degree $2p^{n}$. 
Then $\mathcal{O}_{K}$ is free over $\mathfrak{A}_{K/\Q}$ if and only if the ramification index 
of $p$ in $K/\Q$ is a power of $p$.
\end{theorem}

Here $\Q(\zeta_{p^{n}})^+$ denotes the maximal totally real subfield $\Q(\zeta_{p^{n}})$.
Using the class number computations of Miller \cite{MR3240817} we obtain the following corollary.

\begin{corollary}
Let $K/\Q$ be a dihedral extension of degree $2p^{n}$ where
$(p,n)$ is $(5,2)$, $(5,3)$, $(7,2)$ or $(11,2)$.
Then $\mathcal{O}_{K}$ is free over $\mathfrak{A}_{K/\Q}$ if and only if 
the ramification index of $p$ in $K/\Q$ is a power of $p$.
\end{corollary}

Similar but more complicated results hold when $p=2$ or $3$ (see \S \ref{sec:dihedral}).

If $G$ is non-abelian and non-dihedral such that $\text{Cl}(\Z[G])$ is trivial, then a result of End\^{o} and Hironaka \cite{MR519042} shows that $G$ is isomorphic to $A_4$, $S_4$ or $A_5$; the converse was already shown by Reiner and Ullom \cite{MR0367043}. 
The main results of the present article will be necessary and sufficient conditions for $\mathcal{O}_{K}$
to be free over $\mathfrak{A}_{K/\Q}$ when $G$ is isomorphic to $A_{4}$, $S_{4}$ or $A_{5}$.
The discussion above shows that the main work is in determining when $\mathcal{O}_{K}$ is locally free over
$\mathfrak{A}_{K/\Q}$.
A key ingredient is the notion of hybrid $p$-adic group rings, introduced by Johnston and Nickel \cite{MR3461042}; using this tool it is straightforward to show that $\mathcal{O}_{K}$ is locally free over
$\mathfrak{A}_{K/\Q}$ at $p=3$ when $G$ is isomorphic to $A_{4}$ or $S_{4}$. 

The statements of the following theorems will depend on certain primes of $K$ having given
decomposition or inertia subgroups up to conjugation. We remark that such properties will not depend on which prime of $K$ we choose above a given rational prime. For example, saying that a rational prime is tamely ramified will mean that some, and hence every, prime of $K$ above $p$ is (at most) tamely ramified in $K/\Q$. We shall henceforth abbreviate `at most tamely ramified' to `tamely ramified'.

\begin{theorem}\label{thm:a4intro}
Let $K/\Q$ be a Galois extension with $\Gal(K/\Q) \cong A_{4}$.
Then $\mathcal{O}_{K}$ is free over $\mathfrak{A}_{K/\Q}$ if and only if 
$2$ is tamely ramified or has full decomposition group.
\end{theorem}

The proof of the `if' direction of this result involves the aforementioned tools. 
To prove the converse, we show that if $2$ is wildly ramified and has decomposition group of order $2$ or $4$ then $\mathcal{O}_{K}$ is not locally free over $\mathfrak{A}_{K/\Q}$ at $p=2$. 
This reduces to showing that the lattice $\Ind_{D}^{G}\mathfrak{A}_{K_\mathfrak{P}/\Q_{2}}:=\Z_{2}[G]\otimes_{\Z_{2}[D]}\mathfrak{A}_{K_\mathfrak{P}/\Q_{2}}$ is not free over $\mathfrak{A}_{K/\Q,2}$,
where $\mathfrak{P}$ is a fixed prime above $2$ and $D$ is its decomposition group.
The main theorem used here is Hattori's result \cite{MR0175950} that commutative orders are `clean' 
(see \S \ref{subsec:induction-and-clean}).

\begin{theorem}\label{thm:s4intro}
Let $K/\Q$ be a Galois extension with $G:=\Gal(K/\Q)\cong S_4$. Then $\mathcal{O}_K$ is free over 
$\mathfrak{A}_{K/\Q}$ if and only if one of the following conditions on $K/\Q$ holds:
\begin{enumerate}
\item $2$ is tamely ramified;
\item $2$ has decomposition group equal to the unique subgroup of $G$ of order $12$; 
\item $2$ is wildly and weakly ramified and has full decomposition group; or
\item $2$ is  wildly and weakly ramified, has decomposition group of order $8$ in $G$, and has inertia subgroup equal to the unique normal subgroup of order $4$ in $G$.
\end{enumerate}
\end{theorem}

\begin{theorem}\label{thm:a5intro}
Let $K/\Q$ be a Galois extension with $\Gal(K/\Q) \cong A_{5}$. 
Then $\mathcal{O}_K$ is free over $\mathfrak{A}_{K/\Q}$ 
if and only if all three of the following conditions on $K/\Q$ hold:
\begin{enumerate}
\item $2$ is tamely ramified;
\item $3$ is tamely ramified or is weakly ramified with ramification index $6$; and
\item $5$ is tamely ramified or is weakly ramified with ramification index $10$.
\end{enumerate}
\end{theorem}

In contrast to the proof of Theorem \ref{thm:a4intro}, the proofs of Theorems \ref{thm:s4intro} and \ref{thm:a5intro} use the (updated) implementation in \textsc{Magma} \cite{MR1484478} of the algorithms developed by Bley and Johnston \cite{MR2422318} and by Hofmann and Johnston \cite{MR4136552}.

\subsection*{Acknowledgments} 
This work is part of the author's PhD thesis at the University of Exeter. 
The author wishes to thank his PhD supervisor, Henri Johnston, for his encouragement and valuable comments, Nigel Byott for helpful discussions, and Werner Bley and Tommy Hofmann for their help with the computer calculations in Appendix \S \ref{appendix}.
The author is grateful to the University of Exeter for funding his PhD scholarship.

\subsection*{Notation and conventions}
All rings are assumed to have an identity element and all modules are assumed
to be left modules unless otherwise stated. We denote certain finite groups as follows:
\begin{itemize}
\item $D_{2n}$ is the dihedral group of order $2n$;
\item $Q_{8}$ is the quaternion group of order $8$;
\item $A_{n}$ is the alternating group on $n$ letters;
\item $S_{n}$ is the symmetric group on $n$ letters.
\end{itemize}

Let $K$ be a number field. By a prime of $K$, we mean a non-zero prime ideal of $\mathcal{O}_{K}$.
If $\mathfrak{P}$ be a prime of $K$, we let $K_{\mathfrak{P}}$ denote the completion of $K$ at $\mathfrak{P}$. We say that a prime is tamely ramified if it is at most tamely ramified. 

Let $H$ be a subgroup of a finite group $G$. We denote by $\ncl_G(H)$ the normal closure of $H$ in $G$, defined as the smallest normal subgroup of $G$ containing $H$ or, equivalently, the subgroup generated by all the conjugates of $H$ in $G$.

\section{Associated orders and reduction to the study of local freeness}

\subsection{Lattices and orders}
For further background, we refer the reader to \cite{MR1972204} or \cite{MR632548}. 
Let $R$ be a Dedekind domain with field of fractions $F$.
An $R$-lattice $M$ is a finitely generated torsion-free $R$-module, or equivalently,
a finitely generated projective $R$-module.
Note that any $R$-submodule of an $R$-lattice is again an $R$-lattice.
For any finite-dimensional $F$-vector space $V$, an $R$-lattice in $V$ is a finitely generated $R$-submodule $M$ in $V$.
We define a $F$-vector subspace of $V$ by 
\[
FM := \{ \alpha_{1} m_{1} + \alpha_{2} m_{2} + \cdots + \alpha_{r} m_{r} \mid r \in \Z_{\geq 0}, \alpha_{i} \in F, m_{i} \in M \}
\]
and say that $M$ is a full $R$-lattice in $V$ if $FM=V$. 
We may identify $FM$ with $F \otimes_{R} M$.

Let $A$ be a finite-dimensional $F$-algebra.
An $R$-order in $A$ is a subring $\Lambda$ of $A$ (so in particular has the same unity element as $A$) such that $\Lambda$ is a full $R$-lattice in $A$. 
A $\Lambda$-lattice is a $\Lambda$-module which is also an $R$-lattice.
For $\Lambda$-lattices $M$ and $N$, a homomorphism of $\Lambda$-modules $f: M \rightarrow N$
is called a homomorphism of $\Lambda$-lattices.
 
The following well-known lemma follows from \cite[Exercise 23.2]{MR632548}.

\begin{lemma}\label{lemma:biggerorder}
Let $\Lambda\subseteq\Gamma$ be two $R$-orders in $A$. 
Let $M$ and $N$ be $\Gamma$-lattices and let $f:M\rightarrow N$
be a homomorphism of $\Lambda$-lattices. 
Then $f$ is a homomorphism of $\Gamma$-lattices.
\end{lemma}

\subsection{Associated orders}
Let $\Lambda$ be an $R$-order in a finite-dimensional $F$-algebra $A$.
Let $M$ be a full $R$-lattice in a free $A$-module of rank $1$ (thus $FM \cong A$ as $A$-modules).
The \emph{associated order} of $M$ is defined to be
\[
\mathfrak{A}(A, M) = \{ \lambda \in  A : \lambda M \subseteq M \}.
\]
Note that $\mathfrak{A}(A, M)$ is an $R$-order (see \cite[\S 8]{MR1972204}).
In particular, it is the largest order $\Lambda$ over which $M$ has a structure of $\Lambda$-module. 
The following well-known result says that $\mathfrak{A}(A, M)$ is the only $R$-order in $A$
over which $M$ can possibly be free.

\begin{prop}\label{prop:associatedfree} 
Let $\Lambda$ be an $R$-order in $A$ and let $M$ be a free $\Lambda$-lattice of rank $1$.
Then $FM$ is a free $A$-module of rank $1$ and $\Lambda=\mathfrak{A}(A, M)$.
\end{prop}

\begin{proof}
By hypothesis there exists $\alpha\in M$ such that $M=\Lambda\alpha$ is a free $\Lambda$-module.
Thus $FM=A\alpha$ is free over $A$.
Let $x \in \mathfrak{A}(A, M)$. 
Then $x \alpha \in M = \Lambda \alpha$, so $x\alpha = y\alpha$ for some $y \in \Lambda$.
Since $FM$ is freely generated by $\alpha$, we must have $x=y$.
Hence $\mathfrak{A}(A, M) \subseteq \Lambda$.
The reverse inclusion is trivial and therefore $\Lambda = \mathfrak{A}(A, M)$.
\end{proof}

\begin{remark}\label{rmk:associatedorderring}
Suppose $\Lambda$ is an $R$-order in $A$.
Then clearly $\Lambda \subseteq \mathfrak{A}(A, \Lambda)$.
Moreover, $\mathfrak{A}(A, \Lambda) 1_{A} \subseteq \Lambda$ and so 
$\mathfrak{A}(A, \Lambda) \subseteq \Lambda$.
Therefore $\mathfrak{A}(A, \Lambda) = \Lambda$.
\end{remark}

\subsection{Completion and local freeness}
Let $\mathfrak{p}$ be any maximal ideal of $R$. 
Let $F_{\mathfrak{p}}$ denote the completion 
of $F$ with respect to a valuation defined by $\mathfrak{p}$ and let $R_{\mathfrak{p}}$ be 
the corresponding valuation ring. 
For any $R$-module $M$ we write $M_{\mathfrak{p}}$ for $R_{\mathfrak{p}} \otimes_{R} M$ 
and $V_{\mathfrak{p}}= F_{\mathfrak{p}} \otimes_{F} V$ for any $F$-vector space $V$.
These two notations are consistent as the map
$\lambda \otimes_{\mathcal{O}_{F}} v \mapsto \lambda \otimes_{F} v$
($v \in V$, $\lambda \in \mathcal{O}_{F_{\mathfrak{p}}}$) is an isomorphism
(see \cite[p.\ 93]{MR1215934}).

Let $\Lambda$ be an $R$-order and let $M$ be a $\Lambda$-lattice in some $A$-module $V$.
Then $\Lambda_{\mathfrak{p}}$ is an $R_{\mathfrak{p}}$-order in $A_{\mathfrak{p}}$
and $M_{\mathfrak{p}}$ is a $\Lambda_{\mathfrak{p}}$-lattice  in $V_{\mathfrak{p}}$.
We say that $M$ is locally free over $\Lambda$ if $M_{\mathfrak{p}}$ is free over $\Lambda_{\mathfrak{p}}$
for every $\mathfrak{p}$. 

Let $G$ be a finite group and let $M$ be a full $R[G]$-lattice in a free $A$-module of rank $1$.
Then $R[G] \subseteq \mathfrak{A}(F[G], M)$ and 
$R_{\mathfrak{p}}[G] \subseteq \mathfrak{A}(F_{\mathfrak{p}}[G], M_{\mathfrak{p}}) \cong \mathfrak{A}(F[G], M)_{\mathfrak{p}}$.
Moreover, $M$ is locally free over $\mathfrak{A}(F[G], M)$
if $M_{\mathfrak{p}}$ is free over $\mathfrak{A}(F_{\mathfrak{p}}[G], M_{\mathfrak{p}})$
for every $\mathfrak{p}$. 

\subsection{Associated orders of rings of integers}\label{subsec:ass-orders-rings-of-ints}
Let $K/F$ be a finite Galois extension of number fields and let $G=\Gal(K/F)$.
We consider the behaviour of the associated order 
$\mathfrak{A}_{K/F}:=\mathfrak{A}(F[G],\mathcal{O}_{K})$ with respect to localisation and induction.

Let $\mathfrak{p}$ be a maximal ideal of $\mathcal{O}_{F}$.
Then we have decompositions
\[
K_{\mathfrak{p}} := F_{\mathfrak{p}} \otimes_{F} K
\cong \prod_{\mathfrak{P}' \mid \mathfrak{p}} K_{\mathfrak{P}'}
\quad \textrm{ and } \quad
\mathcal{O}_{K, \mathfrak{p}} := \mathcal{O}_{F_{\mathfrak{p}}} \otimes_{\mathcal{O}_{F}} \mathcal{O}_{K}
\cong \prod_{\mathfrak{P}' \mid \mathfrak{p}} \mathcal{O}_{K_{\mathfrak{P}'}},
\]
where $\{ \mathfrak{P}' \mid \mathfrak{p} \}$ consists of the primes of $\mathcal{O}_{K}$ above $\mathfrak{p}$ (see  \cite[p.\ 109]{MR1215934}).
Fix a prime $\mathfrak{P}$ above $\mathfrak{p}$ and let $D$ be its decomposition group in $G$.
Then as $G$ acts transitively on $\{ \mathfrak{P}' \mid \mathfrak{p} \}$ we have
\[
K_{\mathfrak{p}} \cong \prod_{s \in G/D} sK_{\mathfrak{P}}
\quad \textrm{ and } \quad
\mathcal{O}_{K, \mathfrak{p}} \cong \prod_{s \in G/D} s \mathcal{O}_{K_{\mathfrak{P}}},
\]
where the products run over a system of representatives of the left cosets $G/D$.  
Hence 
\[
\mathcal{O}_{K, \mathfrak{p}}
\cong \mathrm{Ind}_{D}^{G}\mathcal{O}_{K_{\mathfrak{P}}} := 
\mathcal{O}_{F_{\mathfrak{p}}}[G] \otimes_{\mathcal{O}_{F_{\mathfrak{p}}}[D]} 
\mathcal{O}_{K_{\mathfrak{P}}},
\]
and 
\[
\mathfrak{A}_{K/F,\mathfrak{p}}=\mathfrak{A}(F[G],\mathcal{O}_{K})_{\mathfrak{p}} \cong \mathfrak{A}(F_\mathfrak{p}[G],\mathrm{Ind}_{D}^{G}\mathcal{O}_{K_{\mathfrak{P}}}),
\]
where the last isomorphism follows from \cite[Exercise 24.2]{MR632548}, for instance. 
Thus $\mathcal{O}_{K}$ is locally free over $\mathfrak{A}_{K/F}$ at $\mathfrak{p}$ 
if and only if $\mathrm{Ind}_{D}^{G}\mathcal{O}_{K_{\mathfrak{P}}}$ is free over 
$\mathfrak{A}(F_\mathfrak{p}[G],\mathrm{Ind}_{D}^{G}\mathcal{O}_{K_{\mathfrak{P}}})$.

In \S \ref{sec:induction} we will consider the relationship between 
$\mathfrak{A}(F_{\mathfrak{p}}[G],\mathrm{Ind}_{D}^{G}\mathcal{O}_{K_{\mathfrak{P}}})$ and
$\Ind_{D}^{G}\mathfrak{A}_{K_{\mathfrak{P}}/F_{\mathfrak{p}}}$, 
as well as conditions under which the implication
`if $\mathcal{O}_{K_\mathfrak{P}}$ is free over $\mathfrak{A}_{K_\mathfrak{P}/F_\mathfrak{p}}$
then $\mathcal{O}_{K}$ is locally free over $\mathfrak{A}_{K/F}$ at $\mathfrak{p}$' holds.

\subsection*{Notation} 
We henceforth consider the isomorphism $\mathfrak{A}_{K/F,\mathfrak{p}}\cong \mathfrak{A}(F_\mathfrak{p}[G],\mathcal{O}_{K,\mathfrak{p}})$ as an identification. 
In particular, we consider $\mathfrak{A}_{K/F,\mathfrak{p}}$ as an 
$\mathcal{O}_{F,\mathfrak{p}}$-order in $F_{\mathfrak{p}}[G]$.

\subsection{Reduction to the study of local freeness}

For a $\Z$-order $\Lambda$ let $\text{Cl}(\Lambda)$ denote the locally free class group of $\Lambda$
as defined in \cite[(39.12)]{MR892316}; further background material can be found in \cite[\S 49]{MR892316}.
The following proposition underpins the proofs of all the new theorems stated in the introduction.

\begin{prop}\label{prop:trivialclassgroup}
Let $G$ be a finite group such that $\textup{Cl}(\Z[G])$ is trivial
and let $K/\Q$ be a Galois extension with $\Gal(K/\Q) \cong G$. 
Then $\mathcal{O}_{K}$ is free over $\mathfrak{A}_{K/\Q}$ 
if and only if $\mathcal{O}_{K}$ is locally free over $\mathfrak{A}_{K/\Q}$.
\end{prop}

\begin{proof}
One implication is trivial.
By \cite[(49.25)]{MR892316} the inclusion $\Z[G]\subseteq \mathfrak{A}_{K/\Q}$ induces a surjection $\text{Cl}(\Z[G]) \twoheadrightarrow \text{Cl}(\mathfrak{A}_{K/\Q})$, 
and so $\text{Cl}(\mathfrak{A}_{K/\Q})$ is also trivial.
Moreover, by \cite{MR519042} the triviality of $\textup{Cl}(\Z[G])$ 
implies that $G$ must be abelian, dihedral, or isomorphic to $A_{4}$, $S_{4}$ or $A_{5}$
(also see \cite[(50.29)]{MR892316}).
In each of these cases, $\Q[G]$ is isomorphic to a finite direct product of matrix rings over number fields.
Hence by \cite[(51.2)]{MR892316} $\Q[G]$ satisfies the Eichler condition 
(see \cite[(45.4) or \S 51A]{MR892316}). 
Thus the Jacobinski cancellation theorem \cite{MR251063} (also see \cite[(51.24)]{MR892316})
implies that $\mathfrak{A}_{K/\Q}$ has the locally free cancellation property.
The non-trivial implication now follows easily.
\end{proof}

\begin{corollary}\label{cor:A4S5A5-locally-free-implies-free}
Let $K/\Q$ be a Galois extension with $\Gal(K/\Q) \cong A_{4},S_{4}$ or $A_{5}$.
Then $\mathcal{O}_{K}$ is free over $\mathfrak{A}_{K/\Q}$ 
if and only if $\mathcal{O}_{K}$ is locally free over $\mathfrak{A}_{K/\Q}$.
\end{corollary}

\begin{proof}
Let $G=\Gal(K/\Q)$. In each case $\text{Cl}(\Z[G])$ is trivial, as shown in \cite{MR0367043}.
\end{proof}

\section{Review of results relating to local freeness}

\subsection{Freeness results for Galois extensions of $p$-adic fields}\label{padicresults}
Many of the results and definitions of this subsection also hold for local fields of positive characteristic, 
but for simplicity we restrict to the case of $p$-adic fields. We fix a prime number $p$.

\begin{theorem}\label{thm:tamepadic}
Let $K/F$ be a tamely ramified finite Galois extension of $p$-adic fields and let $G=\Gal(K/F)$.
Then $\mathcal{O}_{K}$ is free over $\mathfrak{A}_{K/F}=\mathcal{O}_{F}[G]$.
\end{theorem}

\begin{remark}
Theorem \ref{thm:tamepadic} is usually attributed to Emmy Noether \cite{MR1581331}. 
In fact, as noted in \cite{MR1373957}, she only stated and proved 
the result in the case that $p \nmid |G|$. 
Complete proofs can be found in \cite{MR717033}, \cite{MR825142} and \cite{MR1373957}.
\end{remark}

\begin{theorem}\cite{MR1627831}\label{thm:lettl-abs-abelian-p-adic}
Let $K/F$ be an extension of $p$-adic fields such that $K/\Q_{p}$ is a finite abelian extension. 
Then $\mathcal{O}_{K}$ is free over $\mathfrak{A}_{K/F}$.
\end{theorem}

\begin{theorem}\cite{MR0371857}\label{thm:padicberge}
 Let $K/\Q_p$ be a Galois extension with $\Gal(K/\Q_p)\cong D_{2\ell}$, where $\ell$ is a prime number. Then $\mathcal{O}_K$ is free over $\mathfrak{A}_{K/\Q_p}$.
\end{theorem}

\begin{theorem}\cite{MR299593}\label{thm:padicquaternion}
Let $K/\Q_{p}$ be a Galois extension with $\Gal(K/\Q_p)\cong Q_{8}$.
Then $\mathcal{O}_K$ is free over $\mathfrak{A}_{K/\Q_p}$.
\end{theorem}

\begin{remark}\label{rmk:padicmartinet}
Theorem \ref{thm:padicquaternion} is not explicitly stated in \cite{MR299593}.
However, the proof of Theorem \ref{thm:martinet} given in loc.\ cit.\ also works essentially unchanged
in the setting of Theorem \ref{thm:padicquaternion}.
Alternatively, note that Theorem \ref{thm:padicquaternion} is implied by Theorem \ref{thm:martinet} because
for every $Q_{8}$-extension $L/\Q_{2}$ there exists a $Q_{8}$-extension $K/\Q$ such that
$K_{\mathfrak{P}}=L$, where $\mathfrak{P}$ is the unique prime of $K$ above $2$;
this can be checked using databases of $p$-adic and number fields such as \cite{MR2194887} and \cite{lmfdb}. 
Note that unlike Theorem \ref{thm:martinet}, it is not necessary to assume that the extension 
in question is wildly ramified, thanks to  Theorem \ref{thm:tamepadic}.
\end{remark}

\begin{theorem}\cite{MR748000}\label{thm:padicjaulent}
Let $p,n$ and $r$ be positive integers such that $p$ is an odd prime, $n$ divides $p-1$
and $r$ is a primitive $n$th root modulo $p$.
Let $G$ be the metacyclic group with the following structure:
\begin{equation}\label{eqmetacyclic}
G= \langle x,y:x^{p}=1,y^{n}=1,yxy^{-1}=x^{r} \rangle \cong C_{p} \rtimes C_{n}.
\end{equation}
Let $K/\Q_{p}$ be a Galois extension with $\Gal(K/\Q_p) \cong G$. 
Then $\mathcal{O}_{K}$ is free over $\mathfrak{A}_{K/\Q_p}$.
\end{theorem}

\begin{remark}
In the special case $n=2$, the group $G$ of \eqref{eqmetacyclic} is dihedral of order $2p$. 
\end{remark}

Let $K/F$ be a Galois extension of $p$-adic fields and let $G=\Gal(K/F)$.
We recall that for an integer $t \geq -1$ the $t$-th ramification group is defined to be
\[
G_{t} := \{ \sigma \in G : v_{K}(\sigma(x)-x)\geq t+1\text{ }\forall x\in \mathcal{O}_K \},
\]
where $v_{K}$ is the normalized valuation on $K$ (i.e.\ with image $\Z$). When it is not obvious which extension we are referring to we will use the notation `$G_t(K/F)$' or similar.
Thus $K/F$ is unramified if and only if $G_{0}$ is trivial and is tamely ramified if and only if $G_{1}$ is trivial. 
We say that the extension is weakly ramified if $G_{2}$ is trivial.

\begin{theorem}\cite{MR3411126}\label{thm:weak}
Let $K/F$ be a weakly ramified finite Galois extension of $p$-adic fields and let $G=\Gal(K/F)$.
Then $\mathcal{O}_{L}$ is free over $\mathfrak{A}_{K/F}$.
Moreover, if $K/F$ is both wildly and weakly ramified then
$\mathfrak{A}_{K/F}=\mathcal{O}_K[G][\pi_{K}^{-1}\textup{Tr}_{G_0}]$, 
where $\pi_{K}$ is a uniformiser of $\mathcal{O}_{K}$ and 
$\textup{Tr}_{G_0}=\sum_{\gamma\in G_{0}}\gamma$ is the sum of the elements of the inertia group $G_{0}$.
\end{theorem}

For a subgroup $H$ of $G$ define $\Tr_{H}=\sum_{h\in H}h \in F[G]$ and 
$e_{H}=\frac{1}{|H|}\Tr_{H} \in F[G]$. Note that $e_{H}$ is an idempotent.
We say that $K/F$ is almost-maximally ramified if $e_{H} \in \mathfrak{A}_{K/F}$
for every subgroup $H$ of $G$ such that $G_{t+1} \subseteq H \subseteq G_{t}$
for some $t \geq 1$.

\begin{theorem}\cite[Proposition 7]{zbMATH03623711}\label{thm:locallyfreedihedral}
Let $K/F$ be a finite dihedral extension of $p$-adic fields such that $F/\Q_{p}$ is unramified. 
Let $G=\Gal(K/F)$. 
Then $\mathcal{O}_{K}$ is projective over $\mathfrak{A}_{K/F}$ if and only if 
$\mathcal{O}_{K}$ is free over $\mathfrak{A}_{K/F}$ if and only if either
\begin{enumerate}
\item $K/F$ is almost-maximally ramified, in which case 
$\mathfrak{A}_{K/F}=\mathcal{O}_F[G][\{e_{G_t}\}_{t\geq1}]$, or
\item $K/F$ is not almost-maximally ramified, and the inertia subgroup $G_0$ is dihedral of order $2p$, in which case $\mathfrak{A}_{K/F}=\mathcal{O}_F[G][2e_{G_0}]$.
\end{enumerate}
\end{theorem}

\begin{remark}\label{rmk:idempotent}
Let $H$ be a subgroup of $G$ and let $r$ be a positive integer. 
We now show how to determine whether $\frac{1}{r}\Tr_{H}\in \mathfrak{A}_{K/F}$. 
For example, when $r=|H|$ this can be used to check for almost-maximal ramification.
Let $M$ be the subfield of $K$ fixed by $H$. 
We denote by $\mathfrak{D}_{K/M}$ the different of the extension $K/M$ (see \cite[III\S3]{MR554237}) and by $v_{p}(x)$ the $p$-adic valuation of an integer $x$
(thus $v_{p}$ is the restriction of $v_{\Q_p}$ to $\Z$).
\[
\begin{split}
\frac{1}{r}\text{Tr}_H\in \mathfrak{A}_{K/F} & \Longleftrightarrow \frac{1}{r}\text{Tr}_{K/M}(\mathcal{O}_K)\subseteq \mathcal{O}_M\\
  &\Longleftrightarrow \text{Tr}_{K/M}(\mathcal{O}_{K})\subseteq r\mathcal{O}_{M} \\
  &\Longleftrightarrow \mathcal{O}_K \subseteq r\mathcal{O}_M \mathfrak{D}_{K/M}^{-1}\text{ by \cite[III Proposition 7]{MR554237}}\\
  &\Longleftrightarrow \mathfrak{D}_{K/M} \subseteq r\mathcal{O}_K \\
  &\Longleftrightarrow v_K(\mathfrak{D}_{K/M})\geq e(K/\Q_p)v_p(r)\\
  &\Longleftrightarrow \sum_{i=0}^\infty (|G_i(K/M)|-1)\geq e(K/\Q_p) v_p(r)\text{ by \cite[IV Proposition 4]{MR554237}}.\\
\end{split}
\]
\end{remark}

\begin{remark}\label{rmk:weakalmostmaximal}
 From Theorem \ref{thm:padicberge}, Theorem \ref{thm:locallyfreedihedral} and Remark \ref{rmk:idempotent} (also see \cite[Corollaire to Proposition 3]{MR513880}) we deduce that a dihedral extension $K/\Q_p$ of degree $2p$ is either almost-maximally ramified or is weakly and totally ramified.
\end{remark}

\subsection{Local freeness results for Galois extensions of number fields}

\begin{remark}\label{rmk:free-implies-locally-free}
Let $K/F$ be a finite Galois extension of number fields. 
If $\mathcal{O}_{K}$ is free over $\mathfrak{A}_{F/K}$ then it is clear that 
$\mathcal{O}_{K}$ is locally free over $\mathfrak{A}_{F/K}$.
In particular, the analogues of Theorems \ref{thm:leopoldt}, \ref{thm:berge} and \ref{thm:martinet} all hold,
with `locally free' in place of `free'. 
Theorems \ref{thm:abs-abelian-locally-free} and \ref{thm:jaulent} below are generalisations of the first
two of these analogues.
\end{remark}

\begin{theorem}\label{thm:noetherlocallyfree}
Let $K/F$ be a finite Galois extension of number fields.
Let $G=\Gal(K/F)$ and let $\mathfrak{p}$ be a prime of $F$ that is tamely ramified in $K/F$. 
Then $\mathcal{O}_{K, \mathfrak{p}}$ is free over
$\mathfrak{A}_{K/F,\mathfrak{p}} = \mathcal{O}_{F_{\mathfrak{p}}}[G]$.
\end{theorem}

\begin{theorem}\label{thm:abs-abelian-locally-free}\cite{MR1627831}
Let $K/F$ be an extension of number fields such that $K/\Q$ is a finite abelian extension. 
Then $\mathcal{O}_{K}$ is locally free over $\mathfrak{A}_{K/F}$.
\end{theorem}

\begin{remark}
Theorems \ref{thm:noetherlocallyfree} and \ref{thm:abs-abelian-locally-free} are well-known 
consequences of Theorems \ref{thm:tamepadic} and \ref{thm:lettl-abs-abelian-p-adic}, respectively.
See Remark \ref{rmk:conditions-to-be-a-ring} and Corollary \ref{cor:induction-in-free-rank-1-case}, for instance.
\end{remark}

\begin{theorem}\cite{MR748000}\label{thm:jaulent}
Let $K/\Q$ be a Galois extension such that $\Gal(K/\Q)$ is metacyclic of type (\ref{eqmetacyclic}). 
Then $\mathcal{O}_{K}$ is locally free over $\mathfrak{A}_{K/\Q}$.
\end{theorem}

\begin{theorem}\cite[Théorème]{zbMATH03623711}\label{thm:nopadicdihedral}
Let $K/\Q$ be a finite dihedral extension and let $G=\Gal(K/\Q)$.
Let $p$ be an odd rational prime that is wildly ramified in $K/\Q$
and let $N$ be the unique cyclic subgroup of $G$ of index $2$.
Then $\mathcal{O}_{K,p}$ is projective over $\mathfrak{A}_{K/\Q,p}$ if and only if 
$\mathcal{O}_{K,p}$ is free over $\mathfrak{A}_{K/\Q,p}$ if and only if 
one of the following conditions holds:
\begin{enumerate}
\item $p$ is almost-maximally ramified in $K/\Q$ and $G_{1} \subseteq N$, in which case
\[
\mathfrak{A}_{K/\Q,p}=\Z_p[G][\{e_{G_t}\}_{t\geq1}],\text{ }or
\]
\item $p$ is not almost-maximally ramified, $|G_{0}|=2p$ and $[G:G_0] \mid 2$, in which case
\[
\mathfrak{A}_{K/\Q,p}=\Z_p[G][e_{G_0}].
\]
\end{enumerate}

\end{theorem}

\begin{remark}
In fact, Theorem \ref{thm:nopadicdihedral} is \cite[Théorème]{zbMATH03623711}
specialised to the case that $p$ is odd and the base field is $\Q$; the more
general statement is somewhat more complicated.
\end{remark}

\section{Hybrid group rings and applications to local freeness}

\subsection{Hybrid group rings}\label{hybridsection}

Let $R$ be a discrete valuation ring with fraction field $F$ and let $G$ be a finite group. 
Let $M$ be a full $R[G]$-lattice in $F[G]$. 
Note that $R[G]\subseteq \mathfrak{A}(F[G], M)$. 

For a normal subgroup $N$ of $G$ define $e_{N}=\frac{1}{|N|}\sum_{n\in N}n \in F[G]$
to be the central idempotent associated to $N$. 

\begin{prop}\label{prop:smallassociated}
If $N$ is a normal subgroup of $G$ such that $|N| \in R^{\times}$ then 
\begin{enumerate}
\item $R[G] = e_{N} R[G] \times (1-e_{N})R[G]\cong R[G/N] \times (1-e_N)R[G]$,
\item $e_{N}M$ has the structure of a $e_{N}R[G]\cong R[G/N]$-lattice, and 
\item $e_{N}\mathfrak{A}(F[G],M)=\mathfrak{A}(F[G],M)\cap e_{N}F[G]=\mathfrak{A}(e_NF[G],e_{N}M)\cong\mathfrak{A}(F[G/N],e_{N}M)$.
\end{enumerate}
\end{prop}

\begin{proof}
Since $|N| \in R^{\times}$ we have $e_{N} \in R[G]$. 
Moreover, it is straightforward to show that $e_{N}R[G] \cong R[G/N]$.
Thus we have established (i) and (ii), and it remains to prove (iii).

The last isomorphism of (iii) is immediate from (ii). We now prove the first equality, that is, $e_{N}\mathfrak{A}(F[G],M)=\mathfrak{A}(F[G],M)\cap e_{N}F[G]$. Since $e_{N}\in R[G]\subseteq\mathfrak{A}(F[G],M)$, we easily have that
$e_{N}\mathfrak{A}(F[G],M)\subseteq\mathfrak{A}(F[G],M)\cap e_{N}F[G]$.
The other containment follows from the fact that $e_{N}^{2}=e_{N}$, 
hence any element in $\mathfrak{A}(F[G],M)\cap e_{N}F[G]$, with the harmless multiplication
by $e_{N}$, can be written as an element in $e_{N}\mathfrak{A}(F[G],M)$.

We now prove that $\mathfrak{A}(F[G],M)\cap e_{N}F[G]=\mathfrak{A}(e_NF[G],e_{N}M)$. Consider $e_Nx\in \mathfrak{A}(F[G],M)\cap e_NF[G]$ for a certain $x\in F[G]$;
we have to prove that $e_{N}x$ preserves $e_{N}M$. 
Since $e_{N}M\subseteq M$ and $e_{N}x\in \mathfrak{A}(F[G],M)$, 
we have that $e_{N}xe_{N}M\subseteq M$.
Hence $e_{N}xe_{N}M=e_{N}e_{N}xe_{N}M\subseteq e_{N}M$. Conversely, let us consider an element
$e_{N}x\in \mathfrak{A}(e_NF[G],e_NM)$, thus such that $e_{N}xe_{N}M\subseteq e_{N}M$.
We must prove that $e_{N}x\in \mathfrak{A}(F[G],M)$, which is automatic since 
$e_{N}xM=e_{N}e_{N}xM=e_{N}xe_{N}M\subseteq e_{N}M\subseteq M$.
 \end{proof}

We now recall the notion of hybrid group ring introduced in
\cite[\S 2]{MR3461042} and further developed in \cite[\S 2]{MR3749195}.

\begin{definition}\label{def:N-hybrid}
Let $N$ be a normal subgroup of $G$.
We say that $R[G]$ is $N$-hybrid if $|N| \in R^{\times}$ and
$(1-e_N)R[G]$ is a maximal $R$-order in $(1-e_{N})F[G]$. 
\end{definition}

\begin{remark}
The group ring $R[G]$ is a maximal $R$-order if and only if $|G| \in R^{\times}$ if and only if
$R[G]$ is $G$-hybrid, where the first equivalence is given by \cite[(27.1)]{MR632548}.
In this situation, $\mathfrak{A}(F[G], M)=R[G]$ and thus $M$ is free over $\mathfrak{A}(F[G], M)$
by \cite[(18.10)]{MR1972204}.
\end{remark}

\begin{example}\label{ex:A4S4-hybrid-examples}
Let $G=A_{4}$ or $S_{4}$ and let $N$ be its unique normal subgroup of order $4$. 
Then $\Z_{3}[G]$ is $N$-hybrid as shown in \cite[Examples 2.16 and 2.18]{MR3461042}.
Indeed, we have
\[
\Z_{3}[A_{4}] \cong \Z_{3}[C_{3}] \times M_{3 \times 3}(\Z_{3})
\quad \text{and} \quad
\Z_{3}[S_{4}] \cong \Z_{3}[S_{3}] \times M_{3 \times 3}(\Z_{3}) \times M_{3 \times 3}(\Z_{3}),
\]
where $M_{3 \times 3}(\Z_{p})$ is a maximal $\Z_{p}$-order by \cite[(8.7)]{MR1972204}.
\end{example}

\begin{example}\label{ex:dihedral-example}
Let $n$ be an odd positive integer and let $N_{n}$ be the unique subgroup of index $2$ in $D_{2n}$.
Then $\Z_{2}[D_{2n}]$ is $N_{n}$-hybrid as shown in \cite[Example 2.14]{MR3461042}.
\end{example}

\begin{prop}\label{prop:hybridfree}
Suppose $R[G]$ is $N$-hybrid. Then 
\[
\mathfrak{A}(F[G],M) 
= e_{N}\mathfrak{A}(F[G],M) \times (1-e_N)R[G]
\cong\mathfrak{A}(F[G/N],e_NM) \times (1-e_N)R[G].
\]
Moreover, $e_{N}M$ is free over $\mathfrak{A}(F[G/N],e_{N}M)$
if and only if $M$ is free over $\mathfrak{A}(F[G],M)$.
\end{prop}

\begin{proof}
The first claim follows from Proposition \ref{prop:smallassociated}, Definition \ref{def:N-hybrid},
and the fact that $R[G] \subseteq \mathfrak{A}(F[G],M)$.
Since $(1-e_{N})R[G]$ is a maximal $R$-order,
$(1-e_{N})M$ is free over $(1-e_{N})R[G]$ by \cite[(18.10)]{MR1972204}. 
The second claim now follows from the decomposition $M \cong e_{N} M \oplus (1-e_{N})M$.
\end{proof}

\subsection{Applications to local freeness for extensions of number fields}

\begin{prop}\label{prop:hybridnumberfields}
Let $K/F$ be a finite Galois extension of number fields and let $G=\Gal(K/F)$.
Let $p$ be a rational prime and let $\mathfrak{p}$ be a prime of $F$ above $p$.
Let $N$ be a normal subgroup of $G$ such that $p \nmid |N|$ and let
$M$ be the subfield of $K$ fixed by $N$.
Then we have an identification
$e_{N}\mathcal{O}_{K,\mathfrak{p}}=\mathcal{O}_{M,\mathfrak{p}}$.
Moreover, via this identification, the structure of $e_{N}\mathcal{O}_{K,\mathfrak{p}}$
as an $e_{N}\mathcal{O}_{F_\mathfrak{p}}[G]$-module coincides with the structure
of $\mathcal{O}_{M,\mathfrak{p}}$ as an $\mathcal{O}_{F_{\mathfrak{p}}}[G/N]$-module 
under the canonical identification $G/N \cong \Gal(M/F)$.
In particular,
\[
e_{N}\mathfrak{A}(F_{\mathfrak{p}}[G],\mathcal{O}_{K,\mathfrak{p}})
=\mathfrak{A}(F_{\mathfrak{p}}[G],\mathcal{O}_{K,\mathfrak{p}})\cap e_{N}F[G]
\cong \mathfrak{A}(F_{\mathfrak{p}}[G/N],\mathcal{O}_{M,\mathfrak{p}}).
\]
Now further suppose that $\mathcal{O}_{F_\mathfrak{p}}[G]$ is $N$-hybrid. 
Then
\[
\mathfrak{A}_{K/F,\mathfrak{p}}
\cong
\mathfrak{A}_{M/F,\mathfrak{p}} \times (1-e_{N})\mathcal{O}_{F_\mathfrak{p}}[G],
\]
and $\mathcal{O}_{K,\mathfrak{p}}$ is free over $\mathfrak{A}_{K/F,\mathfrak{p}}$
if and only if $\mathcal{O}_{M,\mathfrak{p}}$ is free over $\mathfrak{A}_{M/F,\mathfrak{p}}$.
\end{prop}

\begin{proof}
The claims regarding the identifications are clear.
The remaining claims are then specialisations of Propositions \ref{prop:smallassociated} and \ref{prop:hybridfree}.
\end{proof}

\begin{corollary}\label{cor:hybrid}
Let $K/\Q$ be a finite Galois extension and let $G=\Gal(K/\Q)$.
Let $N$ be a normal subgroup of $G$ and such that 
$G/N$ is abelian or metacyclic of type (\ref{eqmetacyclic}).
Let $p$ be a rational prime.
If $\Z_{p}[G]$ is $N$-hybrid, then $\mathcal{O}_{K,p}$ is free over $\mathfrak{A}_{K/\Q,p}$.
\end{corollary}

\begin{proof}
Let $M$ be the subfield of $K$ fixed by $N$.
Then $\mathcal{O}_{M,p}$ is free over $\mathfrak{A}_{M/\Q,p}$
by Theorem \ref{thm:leopoldt} or Theorem \ref{thm:jaulent}.
The result now follows from Proposition \ref{prop:hybridnumberfields}.
\end{proof}

\begin{remark}
Jaulent \cite{MR748000} developed similar arguments to Corollary \ref{cor:hybrid},
but restricted to the case that $G$ is metacyclic of type (\ref{eqmetacyclic}).
\end{remark}

\subsection{Preliminary results on $A_{4}$ and $S_{4}$-extensions of $\Q$}

\begin{prop}\label{prop:preliminary-A4-S4-result}
Let $K/\Q$ be a Galois extension with $\Gal(K/\Q)\cong A_{4}$ or $S_{4}$.
Then $\mathcal{O}_{K}$ is free over $\mathfrak{A}_{K/\Q}$ 
if and only if $\mathcal{O}_{K,2}$ is free over $\mathfrak{A}_{K/\Q,2}$. 
\end{prop}

\begin{proof}
By Corollary \ref{cor:A4S5A5-locally-free-implies-free}, it suffices to show that 
$\mathcal{O}_{K,p}$ is free over $\mathfrak{A}_{K/\Q,p}$ for each rational prime $p \geq 3$.
For $p \geq 5$ this follows from Theorem \ref{thm:noetherlocallyfree}.
Let $G=\Gal(K/\Q)$ and let $N$ be its unique normal subgroup of order $4$.
By Example \ref{ex:A4S4-hybrid-examples} the group ring $\Z_{3}[G]$ is $N$-hybrid.
Moreover, $G/N \cong C_{3}$ or $S_{3}$ (note that $S_{3} \cong D_{6}$ is metacyclic of type \eqref{eqmetacyclic}).
Thus by Corollary \ref{cor:hybrid} we have that 
$\mathcal{O}_{K,3}$ is free over $\mathfrak{A}_{K/\Q,3}$.
\end{proof}

\begin{lemma}\label{lem:unique-A4-ext-of-Q2}
There is a unique Galois extension $L/\Q_{2}$ with $\Gal(L/\Q_{2}) \cong A_{4}$.
Moreover, this extension is wildly and weakly ramified, and the inertia subgroup is equal to the unique
(normal) subgroup of order $4$.
\end{lemma}

\begin{proof}
This can easily be checked by, for instance, using the database of $p$-adic fields \cite{MR2194887}. 
Indeed, $L$ is the Galois closure of the extension of $\Q_{2}$
generated by the polynomial $\mathtt{x^{4}+2x^{3}+2x^{2}+2}$. 
\end{proof}

\begin{prop}\label{prop:preliminary-A4-result}
Let $K/\Q$ be a Galois extension with $\Gal(K/\Q) \cong A_{4}$. 
If $2$ is either tamely ramified in $K/\Q$ or has full decomposition group in $\Gal(K/\Q)$,
then $\mathcal{O}_{K}$ is free over $\mathfrak{A}_{K/\Q}$. 
\end{prop}

\begin{proof}
By Proposition \ref{prop:preliminary-A4-S4-result}, it suffices to show that 
$\mathcal{O}_{K,2}$ is free over $\mathfrak{A}_{K/\Q,2}$.
If $2$ is tamely ramified in $K/\Q$ then this follows from Theorem \ref{thm:noetherlocallyfree}.
Now suppose that $2$ has full decomposition group in $G:=\Gal(K/\Q)$. 
Then $2$ is weakly ramified in $K/\Q$ by Lemma \ref{lem:unique-A4-ext-of-Q2}.
Let $\mathfrak{P}$ be the unique prime of $K$ above $2$.
Then 
\[
\mathcal{O}_{K,2} \cong \Ind_{G}^{G}\mathcal{O}_{K_\mathfrak{P}}=\mathcal{O}_{K_\mathfrak{P}}
\quad \text{and} \quad
\mathfrak{A}_{K/\Q,2}\cong \mathfrak{A}_{K_\mathfrak{P}/\Q_{2}},
\]
so the result now follows from Theorem \ref{thm:weak}.
\end{proof}

\section{Leopoldt-type theorems for certain dihedral extensions of $\Q$}\label{sec:dihedral}

We first recall the following theorem of Berg\'e stated the introduction.

\begin{theorem}\cite{MR0371857}
Let $p$ be a prime number and let $K/\Q$ be a dihedral extension of degree $2p$. 
Then $\mathcal{O}_{K}$ is free over $\mathfrak{A}_{K/\Q}$.
\end{theorem}

In the following theorem and corollaries, we consider other dihedral extensions of $\Q$.
For a positive integer $m$,
let $\Q(\zeta_{m})^+$ denote the maximal totally real subfield of $m$th cyclotomic field $\Q(\zeta_{m})$.
If $m$ is odd then $\Q(\zeta_{2m})= \Q(\zeta_{m})$ and so
$\Q(\zeta_{2m})^{+}= \Q(\zeta_{m})^{+}$.

\begin{theorem}\label{thm:dihedralpower}
Let $p$ be a prime and let $n\geq 2$ be an integer.
Let $K/\Q$ be a dihedral extension of degree $2p^{n}$.
Suppose that $p$ is a regular prime such that the class number of $\Q(\zeta_{2p^{n}})^{+}$ is $1$. 
Consider the following assertions:
\begin{enumerate}
\item $\mathcal{O}_{K}$ is free over $\mathfrak{A}_{K/\Q}$;
\item $\mathcal{O}_{K}$ is locally free over $\mathfrak{A}_{K/\Q}$ at $p$;
\item $p$ is almost-maximally ramified in the extension $K/\Q$;
\item the ramification index of $p$ in $K/\Q$ is a power of $p$.
\end{enumerate}
Then we have the following conclusions:
\begin{enumerate}[label=\upshape{(\alph*)}]
\item \textup{(i)} and \textup{(ii)} are equivalent;
\item if $p$ is odd, then \textup{(i)}, \textup{(ii)} and \textup{(iii)} are equivalent;
\item if $p$ is odd, then \textup{(iv)} implies  \textup{(i)}, \textup{(ii)} and \textup{(iii)};
\item if $p\geq 5$, then \textup{(i)}, \textup{(ii)}, \textup{(iii)} and \textup{(iv)} are equivalent.
\end{enumerate}
\end{theorem}

\begin{proof}
Let $G=\Gal (K/\Q)$. 
By \cite[Theorem 10.1]{MR1421575} the condition on the class number of $\Q(\zeta_{2p^{n}})^{+}$
implies that the class number of $\Q(\zeta_{2p^d})^+$ is $1$ for every $d\leq n$. 
This together with the regularity of $p$ implies that the locally free class group
$\text{Cl}(\Z[G])$ is trivial: if $p$ is odd this follows from a special case of the main result of 
\cite[Theorem 1]{MR0357591} (also see \cite[(50.28)]{MR892316}),
if $p=2$ this follows from the results of \cite{MR360531}
(also see \cite[(50.31)]{MR892316} and \cite[(7.39)]{MR632548}).
Therefore
$\mathcal{O}_{K}$ is free over $\mathfrak{A}_{K/\Q}$ if and only if
$\mathcal{O}_{K}$ is locally free over $\mathfrak{A}_{K/\Q}$ by Proposition \ref{prop:trivialclassgroup}.
Note that $\mathcal{O}_{K}$ is locally free over
$\mathfrak{A}_{K/\Q}$ at $\ell$ for every rational prime $\ell \neq 2,p$ by 
Theorem \ref{thm:noetherlocallyfree}.
Moreover, if $p$ is odd then 
Example \ref{ex:dihedral-example} implies that 
$\Z_{2}[G]$ is $N$-hybrid where $N$ is the unique subgroup of $G$ of index $2$,
and so $\mathcal{O}_{K}$ is locally free over
$\mathfrak{A}_{K/\Q}$ at $\ell=2$ by Corollary \ref{cor:hybrid}.
Thus we have proven claim (a).
 
Claim (b) now follows from Theorem \ref{thm:nopadicdihedral}
(note that case (ii) of Theorem \ref{thm:nopadicdihedral} cannot occur when $p$ is odd and $n \geq 2$).
Finally, claims (c) and (d) follow from the characterization of almost-maximal ramification in dihedral extensions given in \cite[Corollaire to Proposition 6]{zbMATH03623711}.
\end{proof}

\begin{remark}
Let $p$ be a prime and let $n$ be a positive integer.
It is well known that $p$ is regular if $p < 37$.
Moreover, by the results of \cite{MR3240817} the class number of $\Q(\zeta_{2p^{n}})^{+}$ is $1$ whenever 
$(p,n)$ is $(2,6)$, $(3,4)$, $(5,3)$, $(7,2)$, $(11,2)$, or the same pairs with a smaller choice of $n \geq 2$. 
Hence the hypotheses of Theorem \ref{thm:dihedralpower} hold for these values.
In particular, we obtain the following corollaries.
\end{remark}

\begin{corollary}
Let $K/\Q$ be a dihedral extension of degree $2 \cdot 3^{n}$ where $n=2,3$ or $4$.
If the ramification index of $3$ in $K/\Q$ is a power of $3$ then 
$\mathcal{O}_{K}$ is free over $\mathfrak{A}_{K/\Q}$.
\end{corollary}

\begin{corollary}
Let $K/\Q$ be a dihedral extension of degree $2p^{n}$ where
$(p,n)$ is $(5,2)$, $(5,3)$, $(7,2)$ or $(11,2)$.
Then $\mathcal{O}_{K}$ is free over $\mathfrak{A}_{K/\Q}$ if and only if 
the ramification index of $p$ in $K/\Q$ is a power of $p$.
\end{corollary}

\begin{remark}
In the proof of Theorem \ref{thm:dihedralpower}, we could have used 
\cite[Théorème]{zbMATH03623711} to establish local freeness at $\ell=2$ 
instead of Example \ref{ex:dihedral-example} and Corollary \ref{cor:hybrid}.
\end{remark}

\section{Review of results on induction of lattices and associated orders}\label{sec:induction}

In this section, we shall give an exposition of Berg\'e's results contained in \cite[\S I]{zbMATH03623711}. 
We include some of the proofs for the convenience of the reader.
The motivation for this section comes from \S \ref{subsec:ass-orders-rings-of-ints}.

\subsection{Associated orders and induction}\label{subsec:induction}
Let $R$ be a Dedekind domain with field of fractions $F$. 
Let $H$ be a subgroup of a finite group $G$ and let $M$ be an $R[H]$-lattice such that 
$FM$ is free of rank $1$ over $F[H]$.

We recall that $\mathrm{Ind}_{H}^{G} M$ is the induced module $R[G] \otimes_{R[H]} M \cong \bigoplus _{s\in G/H} sM$, where on the right hand side we choose a system of representatives in $G$ of the left cosets $G/H$ and the left $R[G]$-module structure is given by the relation $gs=th$ for some coset representative $t$ and $h\in H$.
We wish to understand the relationship between $\mathfrak{A}(F[G], \mathrm{Ind}_{H}^{G} M)$ and
$\mathrm{Ind}_{H}^{G} \mathfrak{A}(F[H], M)$.
Note that these both contain the group ring $R[G]$.

\begin{prop}\cite[\S 1.3]{zbMATH03623711}\label{prop:intersection}
We have
\begin{equation*}
\mathfrak{A}(F[G], \mathrm{Ind}_{H}^{G} M)
= \bigcap_{g \in G} g\mathrm{Ind}_{H}^{G} \mathfrak{A}(F[H], M)g^{-1}.
\end{equation*}
\end{prop}

\begin{corollary}\label{cor:condition-to-be-a-ring}
$\mathrm{Ind}_{H}^{G} \mathfrak{A}(F[H], M)$ is a ring
if and only if it is equal to $\mathfrak{A}(F[G], \mathrm{Ind}_{H}^{G} M)$. 
\end{corollary}

\begin{proof}
If $\mathrm{Ind}_{H}^{G} \mathfrak{A}(F[H], M)$ is a ring then $g\mathrm{Ind}_{H}^{G} \mathfrak{A}(F[H], M)g^{-1} = \mathrm{Ind}_{H}^{G} \mathfrak{A}(F[H], M)$ for all $g \in G$ and 
thus $\mathrm{Ind}_{H}^{G} \mathfrak{A}(F[H], M)=\mathfrak{A}(F[G], \mathrm{Ind}_{H}^{G} M)$
by Proposition \ref{prop:intersection}.
Conversely, if $\mathrm{Ind}_{H}^{G} \mathfrak{A}(F[H], M)=\mathfrak{A}(F[G], \mathrm{Ind}_{H}^{G} M)$ 
then the left hand side is a ring since the right hand side is an associated order and thus a ring.
\end{proof}

\begin{remark}\label{rmk:conditions-to-be-a-ring}
In general, $\mathrm{Ind}_{H}^{G} \mathfrak{A}(F[H], M)$ need not be a ring. 
However, it is straightforward to deduce from the above that
$\mathrm{Ind}_{H}^{G} \mathfrak{A}(F[H], M)$ is a ring in the following cases:
\begin{enumerate}
\item there exists a subgroup $K\leq G$ such that $G\cong H\times K$,
\item $H$ is contained in the centre of $G$, or
\item $\mathfrak{A}(F[H], N) = R[H]$.
\end{enumerate} 
\end{remark}

\begin{remark}\label{rmk:inds-containment}
Proposition \ref{prop:intersection} implies that 
$\mathfrak{A}(F[G], \mathrm{Ind}_{H}^{G} M)\subseteq\mathrm{Ind}_{H}^{G} \mathfrak{A}(F[H], M)$. 
Hence $\mathrm{Ind}_{H}^{G} \mathfrak{A}(F[H], M)$ is an
$\mathfrak{A}(F[G], \mathrm{Ind}_{H}^{G} M)$-lattice.
\end{remark}

\begin{prop}\label{prop:induction-in-free-rank-1-case}
If $M$ is free over $\mathfrak{A}(F[H], M)$ then
$\mathrm{Ind}_{H}^{G} M \cong \mathrm{Ind}_{H}^{G} \mathfrak{A}(F[H], M)$
as $\mathfrak{A}(F[G], \mathrm{Ind}_{H}^{G} M)$-lattices.
\end{prop}

\begin{proof}
Since $R[H] \subseteq \mathfrak{A}(F[H], M)$ and $M$ is free (necessarily of rank $1$) over $\mathfrak{A}(F[H], M)$, 
we see that $M$ and $\mathfrak{A}(F[H], M)$ are isomorphic as $R[H]$-lattices.
Extension of scalars gives an isomorphism 
$\Ind_{H}^{G}M \cong \mathrm{Ind}_{H}^{G} \mathfrak{A}(F[H], M)$ of $R[G]$-lattices.
By Lemma \ref{lemma:biggerorder} this is also an isomorphism of 
$\mathfrak{A}(F[G], \mathrm{Ind}_{H}^{G} M)$-lattices.
\end{proof}

\begin{corollary}\label{cor:induction-in-free-rank-1-case}
Suppose that $M$ is free over $\mathfrak{A}(F[H], M)$.
If
\begin{enumerate}
\item $\mathrm{Ind}_{H}^{G} \mathfrak{A}(F[H], M)$ 
is free over $\mathfrak{A}(F[G], \mathrm{Ind}_{H}^{G} M)$,
\item $\mathrm{Ind}_{H}^{G} \mathfrak{A}(F[H], M)=\mathfrak{A}(F[G], \mathrm{Ind}_{H}^{G} M)$, or
\item $\mathrm{Ind}_{H}^{G} \mathfrak{A}(F[H], M)$ is a ring,
\end{enumerate}
then $\mathrm{Ind}_{H}^{G} M$ is free over $\mathfrak{A}(F[G], \mathrm{Ind}_{H}^{G} M)$.
\end{corollary}

\begin{proof}
In case (i) this follows immediately from  Proposition \ref{prop:induction-in-free-rank-1-case}. 
Clearly, (ii) $\Rightarrow$ (i). Moreover, (ii) $\Leftrightarrow$ (iii) by Corollary \ref{cor:condition-to-be-a-ring}.
\end{proof}

The following is a partial converse to Corollary \ref{cor:induction-in-free-rank-1-case}(i).

\begin{prop}\cite[Proposition 2]{zbMATH03623711}\label{prop:projproj}
If $\mathrm{Ind}_{H}^{G} M$ is a projective $\mathfrak{A}(F[G], \mathrm{Ind}_{H}^{G} N)$-lattice, 
then $M$ is a projective $\mathfrak{A}(F[H], M)$-lattice.
\end{prop}

If $H$ is normal in $G$ then we define
$\mathfrak{A}^{*}(M)=\bigcap_{g\in G}g\mathfrak{A}(F[H], M)g^{-1}$.

\begin{prop}\cite[Proposition 3]{zbMATH03623711}\label{prop:Berge-prop-3}
Suppose that $H$ is normal in $G$. Then 
\begin{enumerate}
\item $\mathfrak{A}^{*}(M)$ is an $R$-order in $F[H]$,
\item $\mathfrak{A}(F[G], \mathrm{Ind}_{H}^{G} M)=\mathrm{Ind}_{H}^{G}\mathfrak{A}^{*}(M)$, and 
\item $\mathrm{Ind}_{H}^{G} M$ is a projective $\mathfrak{A}(F[G], \mathrm{Ind}_{H}^{G} M)$-lattice 
if and only if $M$ is a projective $\mathfrak{A}^{*}(M)$-lattice.
\end{enumerate} 
\end{prop}

\subsection{Clean orders and induction}\label{subsec:induction-and-clean}
Let $R$ be a discrete valuation ring with field of fractions $F$ of characteristic zero
and suppose that the residue field of $R$ is finite. 

\begin{definition}
Let $\Lambda$ be an $R$-order in a finite dimensional semisimple $F$-algebra $A$.
Then $\Lambda$ is said to be \emph{clean} if it has the following property:
if $M$ is a projective $\Lambda$-lattice such that $FM$ is free over $A$ then $M$ is free over $\Lambda$.
\end{definition}

\begin{theorem}[Hattori]\label{thm:Hattori}
Commutative $R$-orders in finite-dimensional semisimple $F$-algebras are clean. 
\end{theorem}

\begin{proof}
See \cite{MR0175950} or \cite[IX Corollary 1.5]{MR0283014}.
\end{proof}


\begin{prop}\cite[Corollaire to Proposition 3]{zbMATH03623711}\label{prop:strongequivalence}
Let $H$ be a normal abelian subgroup of a finite group $G$ and let $M$ be an $R[H]$-lattice such that 
$FM$ is free of rank $1$ over $F[H]$.
Then the following are equivalent:
 \begin{enumerate}
\item $\mathrm{Ind}_{H}^{G}M$ is projective over $\mathfrak{A}(F[G], \mathrm{Ind}_{H}^{G} M)$;
\item $\mathrm{Ind}_{H}^{G}M$ is free over $\mathfrak{A}(F[G], \mathrm{Ind}_{H}^{G} M)$;
\item $\mathrm{Ind}_{H}^{G} \mathfrak{A}(F[H], M)$ is a ring
and $\mathrm{Ind}_{H}^{G}M$ is free over $\mathrm{Ind}_{H}^{G} \mathfrak{A}(F[H], M)$;
\item $M$ is free over $\mathfrak{A}(F[H], M)$ and $\mathfrak{A}^{*}(M)=\mathfrak{A}(F[H], M)$.
\end{enumerate}
\end{prop}

\begin{proof}
(i)$\Rightarrow$(iv). 
By Proposition \ref{prop:Berge-prop-3}(iii), $M$ is projective over $\mathfrak{A}^{*}(M)$.
Moreover, $\mathfrak{A}^{*}(M)$ is a clean order by Theorem \ref{thm:Hattori}
and thus $M$ is in fact free over $\mathfrak{A}^{*}(M)$. 
Hence $\mathfrak{A}^{*}(M)=\mathfrak{A}(F[H], M)$ by 
Proposition \ref{prop:associatedfree}.
 
(iv)$\Rightarrow$(iii). We have
$\Ind^{G}_{H}\mathfrak{A}(F[H], M)=\Ind^{G}_{H}\mathfrak{A}^{*}(M)=\mathfrak{A}(F[G], \mathrm{Ind}_{H}^{G} M)$, where the second equality is Proposition \ref{prop:Berge-prop-3}(ii).
Thus $\Ind^{G}_{H}\mathfrak{A}(F[H], M)$ is a ring by Corollary \ref{cor:condition-to-be-a-ring}. 
Hence $\mathrm{Ind}_{H}^{G}M$ is free over $\mathrm{Ind}_{H}^{G} \mathfrak{A}(F[H], M)$ by 
Corollary \ref{cor:induction-in-free-rank-1-case}(iii).

(iii)$\Rightarrow$(ii). This follows from Corollary \ref{cor:condition-to-be-a-ring}.
 
(ii)$\Rightarrow$(i). This follows from the general fact that every free module is projective.
\end{proof}


\section{Induction for orders of a certain structure}\label{inductionstructure}

Let $R$ be a discrete valuation ring with field of fractions $F$ of characteristic zero and suppose that the residue field of $R$ is finite.
Let $G$ be a finite group and let $H$ be a subgroup of $G$. In \S\ref{sec:induction} we reviewed some general induction properties of the associated order $\mathfrak{A}(F[H], M)$ (with weaker hypotheses on $R$ for some results). In this section, we prove new results concerning inductions of orders of a certain form and then consider arithmetic applications such as the study of weakly ramified extensions.

Let $\pi$ be a uniformizer of $R$.
For a subgroup $P$ of $G$, let $\ncl_G(P)$ denote the normal closure of $P$ in $G$ and let
$\Tr_{P} = \sum_{k\in P}k \in R[G]$.  

\begin{theorem}\label{thm:inductiononegenerator}
Let $M$ be an $R[H]$-lattice such that $FM$ is free of rank $1$ over $F[H]$.
Suppose that there exist a positive integer $n$ and a subgroup $P$ of $H$ such that
\[
\mathfrak{A}(F[H], M)=R[H]+{\pi^{-n}}R[H]\textup{Tr}_P.
\] 
Then the following statements hold:
\begin{enumerate}
\item $\mathrm{Ind}_{H}^{G} \mathfrak{A}(F[H], M)=R[G]+\pi^{-n}R[G]\textup{Tr}_P$.
\item $\mathfrak{A}(F[G], \mathrm{Ind}_{H}^{G} M)=R[G]+\pi^{-n}R[G]\textup{Tr}_{\ncl_G(P)}$.
\item $\mathrm{Ind}_{H}^{G} \mathfrak{A}(F[H], M)$ is a ring if and only if $P$ is normal in $G$.
\item  If $P$ is normal in $G$
and $M$ is free over $\mathfrak{A}(F[H], M)$
then $\mathrm{Ind}_{H}^{G} \mathfrak{A}(F[H], M)$
is free over $\mathfrak{A}(F[G], \mathrm{Ind}_{H}^{G} M)$.
\item If $H$ is abelian and normal in $G$ and $\mathrm{Ind}_{H}^{G} M$ is projective over $\mathfrak{A}(F[G], \mathrm{Ind}_{H}^{G} M)$, then $P$ is normal in $G$.
\end{enumerate}
\end{theorem}

\begin{proof}
Note that if $\{h\in G/H\}$ and $\{k\in H/P\}$ are two sets of left cosets representatives,
then $\{hk\}$ is a set of left coset representatives for $G/P$. Thus we have
\[
\begin{split}
\mathrm{Ind}_{H}^{G} \mathfrak{A}(F[H], M)&=\mathrm{Ind}_{H}^{G}\left(R[H]+\left\{{\pi^{-n}}\left(\sum_{k\in H/P}a_{k}k\right):a_{k}\in R\right\}\cdot\textup{Tr}_{P}\right)\\
 &=R[G]+\left\{\sum_{h\in G/H}{\pi^{-n}}h\left(\sum_{k\in H/P}a_{h,k}k\right): a_{h,k}\in R\right\}\cdot\textup{Tr}_{P}\\
 &=R[G]+\left\{{\pi^{-n}}\left(\sum_{h\in G/H}\sum_{k\in H/P}a_{h,k}hk\right):a_{h,k}\in R\right\}\cdot\textup{Tr}_{P}\\
 &=R[G]+{\pi^{-n}}R[G]\textup{Tr}_{P},
\end{split}
\]
which proves (i). Moreover, we have
\[
\begin{split}
\mathrm{Ind}_{H}^{G} \mathfrak{A}(F[H], M)&=R[G]+{\pi^{-n}}R[G/P]\textup{Tr}_{P}\\
&=R[G]+\left\{{\pi^{-n}}\left(\sum_{h\in G/P}a_hh\right) : a_h\textup{ is a representative of }R/(\pi^{n})\right\}\cdot\textup{Tr}_{P}\\
&={\pi^{-n}}\left\{\sum_{\gamma\in G}a_\gamma\gamma\in R[G] : 
\gamma_{1}^{-1}\gamma_{2} \in P\Rightarrow a_{\gamma_{1}}\equiv a_{\gamma_{2}}\bmod{(\pi^{n})}\right\}.
\end{split}
\]
Thus for every $g\in G$, we have
\[
\begin{split}
g\mathrm{Ind}_{H}^{G} \mathfrak{A}(F[H], M)g^{-1}
&= R[G]+{\pi^{-n}}R[G/gPg^{-1}]\textup{Tr}_{gPg^{-1}}\\
&={\pi^{-n}}\left\{\sum_{\gamma\in G} a_{\gamma}\gamma \in R[G] : \gamma_{1}^{-1}\gamma_{2}\in gPg^{-1}\Rightarrow a_{\gamma_{1}}\equiv a_{\gamma_{2}}\bmod{(\pi^{n})}\right\}.
 \end{split}
\]
Now by Proposition \ref{prop:intersection}, we have that
\begin{equation}\label{eq:ncl}
\begin{split}
\mathfrak{A}(F[G], \mathrm{Ind}_{H}^{G} M)
&= \bigcap_{g \in G} g\mathrm{Ind}_{H}^{G} \mathfrak{A}(F[H], M)g^{-1}\\
&= {\pi^{-n}}\left\{\sum_{\gamma\in G}a_\gamma\gamma\in R[G] : \gamma_{1}^{-1}\gamma_{2}\in 
\bigcup_{g \in G} gPg^{-1}
\Rightarrow a_{\gamma_{1}}\equiv a_{\gamma_{2}}\bmod{(\pi^{n})} \right\} \\ 
&={\pi^{-n}}\left\{\sum_{\gamma\in G}a_\gamma\gamma\in R[G] : \gamma_{1}^{-1}\gamma_{2}\in 
\ncl_G(P)
\Rightarrow a_{\gamma_{1}}\equiv a_{\gamma_{2}}\bmod{(\pi^{n})} \right\} \\
&=R[G]+ {\pi^{-n}}R[G/\ncl_G(P)]\textup{Tr}_{\ncl_G(P)},\\
&=R[G]+ {\pi^{-n}}R[G]\textup{Tr}_{\ncl_G(P)},
\end{split}
\end{equation}
which proves (ii).

By Corollary \ref{cor:condition-to-be-a-ring}, $\mathrm{Ind}_{H}^{G} \mathfrak{A}(F[H], M)$ is a ring if and only if it is equal to $\mathfrak{A}(F[G], \mathrm{Ind}_{H}^{G} M)$, which by (i) and (ii) is true
if and only if $P=\ncl(P)$. This proves (iii).
Part (iv) follows from (iii) and Corollary \ref{cor:induction-in-free-rank-1-case}.
Part (v) follows from (iii) and Proposition \ref{prop:strongequivalence}.
\end{proof}

\begin{remark}
The system of equations (\ref{eq:ncl}) is justified by the following fact. Let $G$ be a group, let $B$ be any set, let $A$ be a subset of $G$ and let $A'$ be the subgroup of $G$ generated by $A$. Then
from the description of the elements of $A'$ in terms of products of elements of $A$ and their
inverses, we have 
\begin{multline*}
\left\{\{a_\gamma\}_{\gamma\in G}\in \textstyle{\prod_{\gamma \in G}} \, B:\gamma_1^{-1}\gamma_2\in A\Rightarrow a_{\gamma_1}=a_{\gamma_2}\right\} \\
=\left\{\{a_\gamma\}_{\gamma\in G}\in  \textstyle{\prod_{\gamma \in G}} B : \gamma_1^{-1}\gamma_2\in A'\Rightarrow a_{\gamma_1}=a_{\gamma_2}\right\}. 
\end{multline*}
\end{remark}

We have the following application to the understanding of 
local freeness in weakly ramified extensions of number fields.

\begin{corollary}\label{cor:inertianormal}
Let $K/F$ be a finite Galois extension of number fields with Galois group $G$ and let $\mathfrak{P}|\mathfrak{p}$ be two primes of $K/F$ such that $K_\mathfrak{P}/F_\mathfrak{p}$ is wildly and weakly ramified. 
\begin{enumerate}
\item If the inertia group $G_{0}=G_{0}(\mathfrak{P}|\mathfrak{p})$ is normal in $G$ then
$\mathcal{O}_{K,\mathfrak{p}}$ is free over $\mathfrak{A}_{K/F,\mathfrak{p}}$.
\item Assume further that the decomposition group $D=D(\mathfrak{P}|\mathfrak{p})$ is abelian and normal in $G$. Then $G_{0}$ is normal in $ G$ if and only if $\mathcal{O}_{K,\mathfrak{p}}$ is free over $\mathfrak{A}_{K/F,\mathfrak{p}}$.
\end{enumerate}
\end{corollary}

\begin{proof}
Let $\pi$ be any uniformizer of $\mathcal{O}_{F_{\mathfrak{p}}}$.
Then by Theorem \ref{thm:weak} we have
\[
\mathfrak{A}(F_{\mathfrak p}[D],\mathcal{O}_{K_{\mathfrak{P}}})
=\mathfrak{A}_{K_\mathfrak{P}/F_\mathfrak{p}}
=\mathcal{O}_{F_\mathfrak{p}}[D][{\pi^{-1}}\textup{Tr}_{G_0}]
=\mathcal{O}_{F_\mathfrak{p}}[D]+{\pi^{-1}}\mathcal{O}_{F_\mathfrak{p}}[D]\textup{Tr}_{G_0}.
 \] 
Hence claims (i) and (ii) follow from parts (iv) and (v) of Theorem \ref{thm:inductiononegenerator}.
\end{proof}

We now prove the following generalisation of Theorem \ref{thm:inductiononegenerator}.

\begin{theorem}\label{thm:indringstar}
Let $M$ be an $R[H]$-lattice such that $FM$ is free of rank $1$ over $F[H]$. Suppose that there exist integers $0=n_0<n_1<\cdots <n_r$ and subgroups
\[
\{e\}=P_0\subsetneq P_1\subsetneq P_2\subsetneq \cdots \subsetneq P_r \subseteq H\subseteq G
\] 
such that 
\begin{equation}\label{eq:assoc-order-with-traces}
\mathfrak{A}(F[H], M)=\sum_{i=0}^r\pi^{-n_i}R[H]\textup{Tr}_{P_i}. 
\end{equation}
Then the following statements hold:
\begin{enumerate}
\item $\mathrm{Ind}_{H}^{G} \mathfrak{A}(F[H], M)=\sum_{i=0}^r\pi^{-n_i}R[G]\textup{Tr}_{P_i}$.
\item $\mathfrak{A}(F[G], \mathrm{Ind}_{H}^{G} M)=\sum_{i=0}^r\pi^{-n_i}R[G]\textup{Tr}_{\ncl_G(P_i)}$.
\item $\mathrm{Ind}_{H}^{G} \mathfrak{A}(F[H], M)$ is a ring if and only if $P_i$ is normal in $ G$ for every $i$.
\item  If $P_i$ is normal in $ G$ for every $i$
and $M$ is free over $\mathfrak{A}(F[H], M)$
then $\mathrm{Ind}_{H}^{G} \mathfrak{A}(F[H], M)$
is free over $\mathfrak{A}(F[G], \mathrm{Ind}_{H}^{G} M)$.
\item If $H$ is abelian and normal in $G$ and $\mathrm{Ind}_{H}^{G} M$ is projective over $\mathfrak{A}(F[G], \mathrm{Ind}_{H}^{G} M)$, then $P_i$ is normal in $ G$ for every $i$.
\end{enumerate}
\end{theorem}

\begin{proof}
The proof of part (i) is exactly as for Theorem \ref{thm:inductiononegenerator}(i).

We already know from Theorem \ref{thm:inductiononegenerator} that (ii) holds if $r=1$. 
So suppose that $r>1$.
Note that, since each $\ncl_G(P_i)$ is normal in $G$, for each $g \in G$ we have that
\[
g^{-1}\pi^{-n_i}\Tr_{\ncl_G(P_i)}g 
= \pi^{-n_i}\Tr_{\ncl_G(P_i)}
=\Tr_{\ncl_G(P_i)/P_i}\pi^{-n_i}\Tr_{P_i}\in \mathrm{Ind}_{H}^{G} \mathfrak{A}(F[H], M),
\]
where $\Tr_{\ncl_G(P_i)/P_i}$ is the sum over any fixed choice of coset representatives of
$\ncl_{G}(P_{i})/P_{i}$. 
Hence for each $g\in G$ we have $\pi^{-n_i}\Tr_{\ncl_G(P_i)}\in g\mathrm{Ind}_{H}^{G} \mathfrak{A}(F[H], M)g^{-1}$. Together with Proposition \ref{prop:intersection}, this implies that
$\pi^{-n_i}\Tr_{\ncl_G(P_i)}\in \mathfrak{A}(F[G], \mathrm{Ind}_{H}^{G} M)$. 
Therefore
\[
\mathfrak{A}(F[G], \mathrm{Ind}_{H}^{G} M)\supseteq\sum_{i=0}^r\pi^{-n_i}R[G]\textup{Tr}_{\ncl_G(P_i)}.
\]
It remains to show the reverse containment. First note that
\begin{equation}\label{eq:ind-contain-eq-base-case}
\mathfrak{A}(F[G], \mathrm{Ind}_{H}^{G} M)
\subseteq \mathrm{Ind}_{H}^{G} \mathfrak{A}(F[H], M)
=\sum_{i=0}^r\pi^{-n_i}R[G]\textup{Tr}_{P_i},
\end{equation}
where the containment follows from Remark \ref{rmk:inds-containment} and the equality is part (i). 
Let $\theta \in \mathfrak{A}(F[G], \mathrm{Ind}_{H}^{G} M)$. 
Then we can write $\theta=\sum_{i=0}^{r} \pi^{-n_i}\theta_{i}\Tr_{P_i}$, where $\theta_{i} \in R[G]$ for each $i$. 
For each integer $j$ with $0 \leq j \leq r$, we shall prove that
\[
\theta \in \sum_{i=0}^{r-j-1}\pi^{-n_i}R[G]\textup{Tr}_{P_i}
+\sum_{i=r-j}^r\pi^{-n_i}R[G]\textup{Tr}_{\ncl_G(P_i)}.
\]

We proceed by induction on $j$ and first consider the base case $j=0$. 
We have that
\[
\pi^{n_{r-1}}\theta=\sum_{i=0}^r \pi^{n_{r-1}-n_i}\theta_i\Tr_{P_i}\in R[G]+\pi^{n_{r-1}-n_r}R[G]\Tr_{P_r}.
\]
Also note that for each $g \in G$ we have
\begin{align*}
g^{-1}\pi^{n_{r-1}}\theta g & \in g^{-1}\pi^{n_{r-1}}\mathfrak{A}(F[G], \mathrm{Ind}_{H}^{G} M)g\\
&= \pi^{n_{r-1}}\mathfrak{A}(F[G], \mathrm{Ind}_{H}^{G} M)\\
&\subseteq\pi^{n_{r-1}}\mathrm{Ind}_{H}^{G} \mathfrak{A}(F[H], M)\\
&\subseteq R[G]+\pi^{n_{r-1}-n_r}R[G]\Tr_{P_r}.
\end{align*}
Hence
\[
\pi^{n_{r-1}}\theta\in \bigcap_{g\in G} g\left(R[G]+\pi^{n_{r-1}-n_r}R[G]\Tr_{P_r}\right)g^{-1}
=R[G]+\pi^{n_{r-1}-n_r}R[G]\Tr_{\ncl_G(P_r)},
\]
where the equality follows from the case $r=1$. 
Thus there exists $\alpha \in R[G]$ such that 
\[
 \theta-\pi^{-n_r}\alpha\Tr_{\ncl_G(P_r)}\in \pi^{-n_{r-1}}R[G]\cap \mathrm{Ind}_{H}^{G} \mathfrak{A}(F[H], M)=\sum_{i=0}^{r-1}\pi^{-n_i}R[G]\textup{Tr}_{P_i},
\]
where the equality follows from \eqref{eq:ind-contain-eq-base-case} and the containment $R[G]\Tr_{P_r}\subseteq R[G]\Tr_{P_{r-1}}$, which holds since $\Tr_{P_r}=\Tr_{P_r/P_{r-1}}\Tr_{P_{r-1}}$.
This completes the base case $j=0$.

We now proceed with the induction step. 
Suppose our claim is valid for $j-1$, and let us prove it for $j$.
Using the inductive hypothesis and subtracting an appropriate element of
$\sum_{i=j+1}^{r} \pi^{-n_i}R[G]\Tr_{\ncl_G(P_i)}$, we can and do assume without loss of generality that 
\[
\theta=\sum_{i=0}^{r-j} \pi^{-n_i}\theta_{i}\Tr_{P_i}\in \mathfrak{A}(F[G], \mathrm{Ind}_{H}^{G} M),
\]
for some $\theta_{i} \in R[G]$. Hence it remains to show that
\[
\theta \in \sum_{i=0}^{r-j-1} \pi^{-n_i}R[G]\Tr_{P_i}+\pi^{-n_{r-j}}R[G]\Tr_{\ncl_G(P_{r-j})}.
\]
As in the base case, for each $g\in G$ we have
\[\begin{split}
g^{-1}\pi^{n_{r-j-1}}\theta g &\in \pi^{n_{r-j-1}-n_{r-j}}R[G]\cap \pi^{n_{r-j-1}}\mathfrak{A}(F[G], \mathrm{Ind}_{H}^{G} M)\\
&\subseteq \pi^{n_{r-j-1}-n_{r-j}}R[G]\cap \pi^{n_{r-j-1}}\mathrm{Ind}_{H}^{G} \mathfrak{A}(F[H], M)\\
&\subseteq R[G]+\pi^{n_{r-j-1}-n_{r-j}}R[G]\Tr_{P_{r-j}},
\end{split}
\]
so, by the result for $r=1$, we have
\[
\pi^{n_{r-j-1}}\theta\in R[G]+\pi^{n_{r-j-1}-n_{r-j}}R[G]\Tr_{\ncl_G(P_{r-j})}.
\]
Thus there exists $\alpha \in R[G]$ such that 
\[
\theta-\pi^{-n_{r-j}}\alpha\Tr_{\ncl_G(P_{r-j})}\in \pi^{-n_{r-j-1}}R[G]\cap \left( \sum_{i=0}^{r-j} \pi^{-n_i}\theta_i\Tr_{P_i}\right)=\sum_{i=0}^{r-j-1}\pi^{-n_i}R[G]\textup{Tr}_{P_i}.
\]
This concludes the induction step. Therefore we deduce that
\[
\mathfrak{A}(F[G], \mathrm{Ind}_{H}^{G} M)=\sum_{i=0}^{r} \pi^{-n_i}R[G]\textup{Tr}_{\ncl_G(P_i)},
\]
which concludes the proof of part (ii).

We easily see with the same methods that 
\[
\mathrm{Ind}_{H}^{G} \mathfrak{A}(F[H], M)=\mathfrak{A}(F[G], \mathrm{Ind}_{H}^{G} M)
\]
precisely when $P_i=\ncl_G(P_i)$ for every $i$, establishing part (iii).
Part (iv) follows from part (iii) and Corollary \ref{cor:induction-in-free-rank-1-case}. 
Part (v) follows from Proposition \ref{prop:strongequivalence}.
\end{proof}

\begin{remark}
It follows from the proof of Theorem \ref{thm:indringstar} that the subgroups $P_i$ and the numbers $n_i$ are uniquely determined.
Moreover, $P_i$ is normal in $H$ (to see this, induct from $H$ to $H$ and use that $\mathfrak{A}(F[H], N)$ is a ring) and $\pi^{n_i}$ divides the order of $P_{i}$ for all $i$.
\end{remark}

\section{Leopoldt-type theorems for $A_4$, $S_4$ and $A_5$-extensions of $\Q$}\label{sec:leopoldthmsA4S4A5}

\subsection{Galois module structure of $A_4$-extensions of $\Q$}

In this subsection, we shall prove the following result, which is Theorem \ref{thm:a4intro} stated in the introduction.

\begin{theorem}\label{thm:a4}
Let $K/\Q$ be a Galois extension with $\Gal(K/\Q)\cong A_4$. Then $\mathcal{O}_K$ is free over $\mathfrak{A}_{K/\Q}$ if and only if $2$ is tamely ramified
or has full decomposition group.
\end{theorem}

\begin{remark}
After considering computational evidence, in \cite[\S 8]{MR2422318} the authors raised the question
of whether it is always the case that $\mathcal{O}_K$ is free over $\mathfrak{A}_{K/\Q}$ for every 
$A_{4}$-extension $K/\Q$. Theorem \ref{thm:a4} shows that this is false.
Indeed, one can use the database of number fields \cite{lmfdb} to verify that every possible decomposition
group of $2$ of even order can be realised by an $A_4$-extension $K/\Q$ in which $2$ is wildly ramified.
\end{remark}

\begin{proof}[Proof of Theorem \ref{thm:a4}]
We already showed the `if' direction in Proposition \ref{prop:preliminary-A4-result}. 
Now we prove that if $2$ is wildly ramified in 
$K/\Q$ and does not have full decomposition group then $\mathcal{O}_{K}$ is not (locally) free (at $2$) over $\mathfrak{A}_{K/\Q}$. 
Let $V_4$ denote the unique normal subgroup of $G:=\Gal(K/\Q)\cong A_{4}$ of order $4$, which is isomorphic to $C_{2} \times C_{2}$. 
Recall that $A_{4} \cong V_{4} \rtimes C_{3}$ and we have the following lattice of the subgroups of $A_4$ up to conjugacy (see, for instance, the GroupNames database \cite{groupnames}).
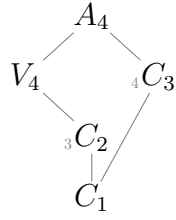
\[
\begin{tikzpicture}[scale=1.0,sgplattice]
  \node[char] at (1,0) (1) {\gn{C1}{C_1}};
  \node at (1,0.803) (2) {\gn{C2}{C_2}};
  \node at (1.88,1.61) (3) {\gn{C3}{C_3}};
  \node[char] at (0.125,1.61) (4) {\gn{V4}{V_{4}}};
  \node[char] at (1,2.41) (5) {\gn{A4}{A_4}};
  \draw[lin] (1)--(2) (1)--(3) (2)--(4) (3)--(5) (4)--(5);
  \node[cnj=2] {3};
  \node[cnj=3] {4};
\end{tikzpicture}
\]
Here the subscript on the left denotes the number of conjugate subgroups, and is taken to be $1$ when omitted (so that the subgroup is normal). 

Let $\mathfrak{P}$ be a prime of $K$ above $2$, let $D=D(\mathfrak{P}|2)$ be the decomposition group and let $G_0=G_0(\mathfrak{P}|2)$ be the inertia group of $K/\Q$.
From the subgroup lattice, we see that it suffices to analyse the cases in which $2$ is wildly ramified in
$K/\Q$ and $D=V_{4}$ or $D\cong C_{2}$.
More precisely, there are three possibilities for the pair $(D,G_0)$ up to isomorphism: 
$(V_{4},V_{4})$, $(V_{4},C_{2})$ and $(C_{2},C_{2})$. 
Since $D$ is abelian in each of these cases, we have that $\mathcal{O}_{K_\mathfrak{P}}$ is free over $\mathfrak{A}_{K_\mathfrak{P}/\Q_{2}}$ by Theorem \ref{thm:lettl-abs-abelian-p-adic}.
Thus by Proposition \ref{prop:induction-in-free-rank-1-case} we have that 
$\textup{Ind}_{D}^{G}\mathfrak{A}_{K_\mathfrak{P}/\Q_2} \cong \textup{Ind}_{D}^{G} \mathcal{O}_{K_{\mathfrak{P}}} \cong \mathcal{O}_{K,2}$
as $\mathfrak{A}_{K/\Q,2}$-lattices.
Therefore it suffices to show that 
$\textup{Ind}_{D}^{G}\mathfrak{A}_{K_\mathfrak{P}/\Q_2}$ is not free over $\mathfrak{A}_{K/\Q,2}$
in each of the three cases.

First suppose that $D=G_{0}=V_{4}$. 
Then from the database of $p$-adic fields \cite{MR2194887} we see that there are four possible 
extensions $K_{\mathfrak{P}}/\Q_{2}$, each of which has $1$ and $3$ as (lower) ramification jumps.
Let $F$ denote the subfield of $K_{\mathfrak{P}}$ fixed by $G_{2}$.
Then by Remark \ref{rmk:idempotent} we have
$e_{G_2} \in \mathfrak{A}_{K_\mathfrak{P}/\Q_{2}}$ since
\[
 \sum_{i=0}^{\infty} (|G_i(K_\mathfrak{P}/F)|-1)
 =1+1+1+1=4\geq 4=|G_0(K_\mathfrak{P}/\Q_2)|\cdot v_2(|G_2|).
\]
Similarly, we have $e_{V_4}=e_{G_1} \in \mathfrak{A}_{K_\mathfrak{P}/\Q_{2}}$ since
\[
 \sum_{i=0}^{\infty} (|G_i(K_\mathfrak{P}/\Q_2)|-1)
 =3+3+1+1=8\geq 8=|G_0(K_\mathfrak{P}/\Q_2)|\cdot v_2(|V_4|).
\]
Thus $K_\mathfrak{P}/\Q_2$ is almost-maximally ramified, and so by Theorem \ref{thm:locallyfreedihedral}(i) 
we have that
\[
\mathfrak{A}_{K_\mathfrak{P}/\Q_2}=\Z_2[D][e_{G_2},e_{G_1}]=\Z_2[V_4]+\frac{1}{2}\Z_2[V_4]\Tr_{G_2}+\frac{1}{4}\Z_2[V_4]\Tr_{V_4}.
\]
Since $G_2\cong C_2$ is not normal in $G$ and $D$ is both abelian and normal in $G$, 
Theorem \ref{thm:indringstar}(iv)\&(v) imply that 
$\textup{Ind}_D^{G}\mathfrak{A}_{K_\mathfrak{P}/\Q_2}$ is not free over $\mathfrak{A}_{K/\Q,2}$. Hence $\mathcal{O}_{K,2}$ is not free over $\mathfrak{A}_{K/\Q,2}$. 
As an aside, we note that, by Theorem \ref{thm:indringstar}(i)\&(ii), in this case we have
\[
\Ind_{D}^G\mathfrak{A}_{K_\mathfrak{P}/\Q_2}=\Z_2[G]+\frac{1}{2}\Z_2[G]\Tr_{G_2}+\frac{1}{4}\Z_2[G]\Tr_{V_4}\supsetneq
\Z_2[G]+\frac{1}{4}\Z_2[G]\Tr_{V_4}=\mathfrak{A}_{K/\Q,2}.
\]

Now suppose that $D=V_4$ and $G_{0} \cong C_{2}$.
Since $\mathcal{O}_{K_\mathfrak{P}}$ is free over $\mathfrak{A}_{K_\mathfrak{P}/\Q_2}$ and
$G_{0}=G_{1}$ is not dihedral of order $4$, Theorem \ref{thm:locallyfreedihedral} implies that 
\[
\mathfrak{A}_{K_\mathfrak{P}/\Q_2}=\Z_2[D][e_{G_0}]=\Z_2[V_4]+\frac{1}{2}\Z_2[V_4]\Tr_{G_0}.
\]
(Alternatively, we can use the database of $p$-adic fields \cite{MR2194887} 
to check for almost-maximal ramification as in the previous case; 
the ramification jump turns out to be $1$ or $2$). 
Since $D$ is both abelian and normal in $G$ and $G_{0}$ is not normal in $G$,
Theorem \ref{thm:inductiononegenerator}(iv)\&(v) imply that 
$\textup{Ind}_D^{G}\mathfrak{A}_{K_\mathfrak{P}/\Q_2}$ is not free over $\mathfrak{A}_{K/\Q,2}$.
Hence $\mathcal{O}_{K,2}$ is not free over $\mathfrak{A}_{K/\Q,2}$. 
Note that in the next paragraph we shall use the following fact
\[
\Ind_{D}^G\mathfrak{A}_{K_\mathfrak{P}/\Q_2}=\Z_2[G]+\frac{1}{2}\Z_2[G]\Tr_{G_0}
\supsetneq
\Z_2[G]+\frac{1}{2}\Z_2[G]\Tr_{V_4}=\mathfrak{A}_{K/\Q,2},
\]
which follows from Theorem \ref{thm:inductiononegenerator}(i)\&(ii).

Finally suppose $D=G_{0} \cong C_{2}$. Then $\mathfrak{A}_{K_\mathfrak{P}/\Q_2}$ is a $\Z_2$-order in $\Q_2[D]\cong \Q_2[C_2]$ strictly containing $\Z_2[D]$. As there is only one such order, we must have
\[
 \mathfrak{A}_{K_\mathfrak{P}/\Q_2}=\Z_2[D]+\frac{1}{2}\Z_2[D]\Tr_D.
\]
We cannot directly apply Theorem \ref{thm:inductiononegenerator}(iv)\&(v) as in the previous cases, as $D$ is not normal in $G$, but from Theorem \ref{thm:inductiononegenerator}(i)\&(ii) we have that
\[
 \mathrm{Ind}_{D}^{G} \mathfrak{A}_{K_\mathfrak{P}/\Q_2}
=\Z_2[G]+\frac{1}{2}\Z_2[G]\text{Tr}_{D}\supsetneq
 \Z_2[G]+\frac{1}{2}\Z_2[G]\text{Tr}_{V_4}=\mathfrak{A}_{K/\Q,2}.
\]
Once we fix a copy of $C_2$ inside $G$, note that $\mathrm{Ind}_{D}^{G} \mathfrak{A}_{K_\mathfrak{P}/\Q_2}$ and $\mathfrak{A}_{K/\Q,2}$ are exactly the same as in the $(V_{4},C_2)$-case of the previous paragraph, where we already showed that $\mathrm{Ind}_{D}^{G} \mathfrak{A}_{K_\mathfrak{P}/\Q_2}$ is not free over $\mathfrak{A}_{K/\Q,2}$.
Therefore $\mathcal{O}_{K,2}$ is not free over $\mathfrak{A}_{K/\Q,2}$. 
\end{proof}

\subsection{Galois module structure of $S_4$-extensions of $\Q$}

In this subsection, we shall prove the following result, which is Theorem \ref{thm:s4intro} stated in the introduction.

\begin{theorem}\label{thm:s4classification}
Let $K/\Q$ be a Galois extension with $G:=\Gal(K/\Q)\cong S_4$. Then $\mathcal{O}_K$ is free over 
$\mathfrak{A}_{K/\Q}$ if and only if one of the following conditions on $K/\Q$ holds:
\begin{enumerate}
\item $2$ is tamely ramified;
\item $2$ has decomposition group equal to the unique subgroup of $G$ of order $12$;
\item $2$ is wildly and weakly ramified and has full decomposition group; or
\item $2$ is wildly and weakly ramified, has decomposition group of order $8$ in $G$,
and has inertia subgroup equal to the unique normal subgroup of order $4$ in $G$.
\end{enumerate}
\end{theorem}

\begin{proof}
By Proposition \ref{prop:preliminary-A4-S4-result}, 
$\mathcal{O}_K$ is free over $\mathfrak{A}_{K/\Q}$ if and only if
$\mathcal{O}_{K,2}$ is free over $\mathfrak{A}_{K/\Q,2}$.

We first show that if any of conditions (i)--(iv) hold then 
$\mathcal{O}_{K,2}$ is free over $\mathfrak{A}_{K/\Q,2}$.
In case (i), this follows from Theorem \ref{thm:noetherlocallyfree}.
In case (ii), Lemma \ref{lem:unique-A4-ext-of-Q2} shows that $2$ is wildly and weakly
ramified in $K/\Q$ and has inertia group equal to the unique normal subgroup of order $4$ in $G$
(note that $A_{4}$ is the unique subgroup of $S_{4}$ of order $12$).
Therefore in cases (ii), (iii) and (iv), $2$ is wildly and weakly ramified in $K/\Q$ and its
inertia group is normal in $G$, and so the desired result follows from 
Corollary \ref{cor:inertianormal}(i).

It now remains to show that if we are not in any of the cases (i)--(iv) then 
$\mathcal{O}_{K,2}$ is not free over $\mathfrak{A}_{K/\Q,2}$. 
We shall use the \textsc{Magma} implementation of two different algorithms: we use \cite[Algorithm 3.1(6)]{MR2422318} to verify that in four specific 
$S_{4}$-extensions of $\Q_{2}$ the ring of integers is not free over its associated order; 
we use \cite[\S8.5]{MR4136552}, which concerns general lattices in group rings, 
to prove that $\mathcal{O}_{K,2}$ is not free over $\mathfrak{A}_{K/\Q,2}$ in cases when we do have freeness in the corresponding $2$-adic extension.

We have the following lattice of the subgroups of $S_4$ up to conjugacy (see, for instance, the GroupNames database \cite{groupnames}).
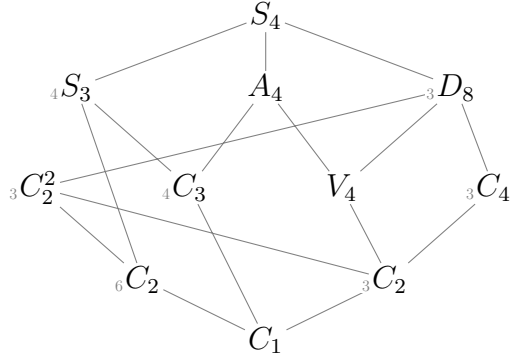
\[
\begin{tikzpicture}[scale=1.0,sgplattice]
  \node[char] at (3.12,0) (1) {\gn{C1}{C_1}};
  \node at (4.75,0.803) (2) {\gn{C2}{C_2}};
  \node at (1.5,0.803) (3) {\gn{C2}{C_2}};
  \node at (2.12,2.02) (4) {\gn{C3}{C_3}};
  \node[char] at (4.12,2.02) (5) {\gn{V4}{V_{4}}};
  \node at (0.125,2.02) (6) {\gn{C2^2}{C_2^2}};
  \node at (6.12,2.02) (7) {\gn{C4}{C_4}};
  \node at (0.625,3.35) (8) {\gn{S3}{S_3}};
  \node at (5.62,3.35) (9) {\gn{D4}{D_8}};
  \node[char] at (3.12,3.35) (10) {\gn{A4}{A_4}};
  \node[char] at (3.12,4.3) (11) {\gn{S4}{S_4}};
  \draw[lin] (1)--(2) (1)--(3) (1)--(4) (2)--(5) (2)--(6) (3)--(6) (2)--(7)
     (3)--(8) (4)--(8) (5)--(9) (6)--(9) (7)--(9) (4)--(10) (5)--(10)
     (8)--(11) (9)--(11) (10)--(11);
  \node[cnj=2] {3};
  \node[cnj=3] {6};
  \node[cnj=4] {4};
  \node[cnj=6] {3};
  \node[cnj=7] {3};
  \node[cnj=8] {4};
  \node[cnj=9] {3};
\end{tikzpicture}
\]
Here the subscript on the left denotes the number of conjugate subgroups, and is taken to be $1$ when omitted (so that the subgroup is normal). 
In particular, the only normal subgroups are $C_{1}$, $V_{4}$, $A_{4}$ and $S_{4}$. 
(Recall that $V_{4} \cong C_{2} \times C_{2}$.)

We fix an isomorphism $G:=\Gal(K/\Q) \cong S_{4}$ and denote by
$A_4$, $D_{8}$, $S_{3}$, $C_{2}^{2}$, $V_{4}$, $C_{4}$
a choice of subgroups of $G$ in 
such a way that whenever there is a containment between choices of conjugates of two such subgroups,
one of the subgroups is in fact contained in the other.

Suppose that $K/\Q$ does not satisfy any of the conditions (i)--(iv).
Let $\mathfrak{P}$ be a prime of $K$ above $2$
and let $D=D(\mathfrak{P}|2)$ be the decomposition group. 
In particular, $2$ is wildly ramified in $K/\Q$ and so $D$ must be of even order.
We cannot have $D=A_{4}$ as this corresponds to case (ii).
Moreover, we cannot have $D=S_{3}$ since the Sylow $2$-subgroups of $D$ are not normal, but
the wild inertia subgroup $G_{1}$ must be normal in $D$.

Suppose that $D=S_{4}$. Since we are not in case (iii), this implies that $2$ is wildly but not weakly ramified in $K/\Q$. From the database of $p$-adic fields \cite{MR2194887}, we see that there are four possibilities
for the completed extension $K_{\mathfrak{P}}/\Q_{2}$.
By using the updated \textsc{Magma} implementation of \cite[Algorithm 3.1(6)]{MR2422318}
(see \S\ref{codeS4Q2} for details), which is based on that of \cite[\S 4.2]{MR2564571},
we can verify that $\mathcal{O}_{K_\mathfrak{P}}$ is not free over $\mathfrak{A}_{K_\mathfrak{P}/\Q_2}$
in either of these cases (for the details on the implementation see \S\ref{codeS4Q2}). Since the decomposition group is full (i.e.\ $D=G$), 
this immediately implies that $\mathcal{O}_{K,2}$ is not free over $\mathfrak{A}_{K/\Q,2}$.

Now suppose that $D\cong D_8$; without loss of generality, we can and do assume that $D=D_8$. 
We first treat the case in which $\mathcal{O}_{K_\mathfrak{P}}$ is not free over
$\mathfrak{A}_{K_{\mathfrak{P}}/\Q_2}$. 
Then $\mathcal{O}_{K_\mathfrak{P}}$ is not projective over $\mathfrak{A}_{K_\mathfrak{P}/\Q_2}$
by Theorem \ref{thm:locallyfreedihedral}, and so $\mathcal{O}_{K,2}$ is not free over $\mathfrak{A}_{K/\Q,2}$ by Proposition \ref{prop:projproj}.
As an aside, by using the database \cite{MR2194887}, Theorem \ref{thm:locallyfreedihedral} 
and Remark \ref{rmk:idempotent}, it is straightforward to check that 
$\mathcal{O}_{K_\mathfrak{P}}$ is not free over $\mathfrak{A}_{K_\mathfrak{P}/\Q_2}$ if and only if 
the ramification jumps of $K_\mathfrak{P}/\Q_2$ are $1$, $3$ and $5$.

Therefore in the remaining cases we can and do assume that $\mathcal{O}_{K_\mathfrak{P}}$ is free over $\mathfrak{A}_{K_{\mathfrak{P}}/\Q_2}$, since either $D$ is abelian (in which case we can apply 
Theorem \ref{thm:lettl-abs-abelian-p-adic}) or $D=D_{8}$, in which case the situation in which $\mathcal{O}_{K_\mathfrak{P}}$ is not free over $\mathfrak{A}_{K_{\mathfrak{P}}/\Q_2}$ has already been considered in the previous paragraph. As we wish to show that $\mathcal{O}_{K,2}$ is not free over $\mathfrak{A}_{K/\Q,2}$, by Proposition \ref{prop:induction-in-free-rank-1-case} it suffices to show that 
$\Ind_D^G \mathfrak{A}_{K_{\mathfrak{P}}/\Q_2}$ is not free over $\mathfrak{A}_{K/\Q,2}$.
Our strategy will be to determine the possible ramification groups, use this to understand
the structure of $\mathfrak{A}_{K_{\mathfrak{P}}/\Q_2}$ and then apply Theorem \ref{thm:indringstar} to obtain 
explicit descriptions of $\Ind_D^G \mathfrak{A}_{K_{\mathfrak{P}}/\Q_2}$ and $\mathfrak{A}_{K/\Q,2}$; we will leave this last step to the end as cases with different $D$ will overlap. 

We now return to the case $D=D_{8}$.
We have the following subgroup lattice.
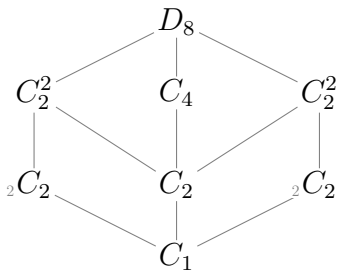
\[
\begin{tikzpicture}[scale=1.0,sgplattice]
  \node[char] at (2,0) (1) {\gn{C1}{C_1}};
  \node[char] at (2,0.953) (2) {\gn{C2}{C_2}};
  \node at (0.125,0.953) (3) {\gn{C2}{C_2}};
  \node at (3.88,0.953) (4) {\gn{C2}{C_2}};
  \node[norm] at (0.125,2.17) (5) {\gn{C2^2}{C_2^2}};
  \node[char] at (2,2.17) (6) {\gn{C4}{C_4}};
  \node[norm] at (3.88,2.17) (7) {\gn{C2^2}{C_2^2}};
  \node[char] at (2,3.12) (8) {\gn{D4}{D_8}};
  \draw[lin] (1)--(2) (1)--(3) (1)--(4) (2)--(5) (3)--(5) (2)--(6) (2)--(7)
     (4)--(7) (5)--(8) (6)--(8) (7)--(8);
  \node[cnj=3] {2};
  \node[cnj=4] {2};
\end{tikzpicture}
\]
We denote the unique normal subgroup of order $2$ in $D_{8}$ by $V_{2}$, and note that 
as a subgroup of $S_{4}$, this is generated by a double transposition and contained in $C_4$. 
(Also note that under $D_8$-conjugation we have three conjugacy classes of subgroups of order $2$ compared to two in the $S_4$-lattice).

Suppose that $K_\mathfrak{P}/\Q_2$ is almost-maximally ramified. 
Then by Theorem \ref{thm:locallyfreedihedral} we have $\mathfrak{A}_{K_\mathfrak{P}/\Q_2}=\Z_2[D]+\sum_{t\geq1}\frac{1}{|G_t|}\Z_2[D]\Tr_{G_t}$ and all quotients of two consecutive different ramification groups are of order $2$ (see the database \cite{MR2194887}, for example). 
We have that $V_2$ must be among the ramification groups since they are all normal in $D_{8}$.
Thus if the ramification index of $K_\mathfrak{P}/\Q_2$ is 2, then $V_2$ is the unique 
ramification group. Otherwise, there is a ramification group of order $4$, which must
be one of $V_4$, $C_2^2$ or $C_4$. Moreover, $D_8$ is a ramification group if and only if the
ramification index of $K_\mathfrak{P}/\Q_2$ is $8$.

Suppose that $K_\mathfrak{P}/\Q_2$ is not almost-maximally ramified. 
Then by Theorem \ref{thm:locallyfreedihedral} we deduce that $G_0\cong C_2^2$ and
$\mathfrak{A}_{K_\mathfrak{P}/\Q_2}=\Z_2[D]+\frac{1}{2}\Z_2[D]\Tr_{G_0}$.
Moreover, by Remark \ref{rmk:idempotent} we must have that $G_2=0$ or $G_2\cong C_2$ and $G_3=0$, but in the latter case the upper ramification jumps are not integral, which is not possible by Hasse-Arf theorem (alternatively just use the database \cite{MR2194887}); hence $K_\mathfrak{P}/\Q_2$ is weakly ramified.
Note that $G_0$ is not equal to $V_4$, otherwise we are in case (iv), hence we can assume $G_0=C_2^2$. 

Now suppose that $D\cong C_4$; without loss of generality, we can and do assume that $D=C_{4}$.
If the ramification index of $K_\mathfrak{P}/\Q_2$ is $2$, then $G_0=V_2$ and by Remark \ref{rmk:idempotent} $K_\mathfrak{P}/\Q_2$ is almost-maximally ramified, and hence $\mathfrak{A}_{K_\mathfrak{P}/\Q_2}=\Z_2[D]+\frac{1}{2}\Z_2[D]\Tr_{G_0}$ (see \cite[Corollaire 3 to Th\'eor\`eme 1]{MR513880}, for example). 
If the ramification index is $4$, then there must be two ramification jumps;
since the upper ramification jumps are integral, Remark \ref{rmk:idempotent} implies that the extension is almost-maximally ramified and so $\frac{1}{2}\Tr_{V_2}$ and $\frac{1}{4}\Tr_{D}$ belong to 
$\mathfrak{A}_{K_\mathfrak{P}/\Q_2}$. Hence
\[
\mathfrak{A}_{K_\mathfrak{P}/\Q_2} = \Z_2[D]+\frac{1}{2}\Z_2[D]\Tr_{V_2}+\frac{1}{4}\Z_2[D]\Tr_{D}, 
\]
where the containment `$\subseteq$' follows from the fact that the right-hand side is the unique maximal order in $\Q_2[D]$ (see \cite[Proposition 5]{MR513880}, for example). 

Finally, note that in the cases $D\cong C_2^2$ or $D\cong C_2$, we already computed $\mathfrak{A}_{K_\mathfrak{P}/\Q_2}$ in the proof of Theorem \ref{thm:a4}.

We will denote by $W_2$ a choice of a subgroup of $S_4$ generated by a transposition and contained in $D_8$.
 We fix the following notation: 
$\langle 1,\frac{1}{2^{n_1}}\Tr_{H_1},\cdots,\frac{1}{2^{n_k}}\Tr_{H_k}\rangle$ is the lattice 
\[
\Z_2[G]+\frac{1}{2^{n_1}}\Z_2[G]\Tr_{H_1}+\cdots+\frac{1}{2^{n_k}}\Z_2[G]\Tr_{H_k}.
\]
Since we have now determined $\mathfrak{A}_{K_\mathfrak{P}/\Q_2}$ in all the remaining cases,
we can use Theorem \ref{thm:indringstar}(i)\&(ii) to determine
$\Ind_D^G\mathfrak{A}_{K_\mathfrak{P}/\Q_2}$ and $\mathfrak{A}_{K/\Q,2}$ as listed in Table \ref{tables4} below.

\begin{table}[htb]
\centering
\caption{}\label{tables4}
\begin{tabular}{|c|c|c|}
\hline
 \textcolor{white}{$\overline{\overline{\underline{\sum}}}$}& $\Ind_D^G\mathfrak{A}_{K_\mathfrak{P}/\Q_2}$ & $\mathfrak{A}_{K/\Q,2}$   \\
\hline
\textcolor{white}{$\overline{\overline{\underline{\sum}}}$}(i)\textcolor{white}{$\overline{\overline{\underline{\sum}}}$} & $\langle 1,\frac{1}{2}\Tr_{V_2},\frac{1}{4}\Tr_{C_4},\frac{1}{8}\Tr_{D_8}\rangle$ &  $\langle 1,\frac{1}{2}\Tr_{V_4},\frac{1}{8}\Tr_{G}\rangle$\\
\hline
\textcolor{white}{$\overline{\overline{\underline{\sum}}}$}(ii)\textcolor{white}{$\overline{\overline{\underline{\sum}}}$} & $\langle 1,\frac{1}{2}\Tr_{V_2},\frac{1}{4}\Tr_{V_4},\frac{1}{8}\Tr_{D_8}\rangle$ &  $\langle 1,\frac{1}{4}\Tr_{V_4},\frac{1}{8}\Tr_{G}\rangle$\\
\hline
\textcolor{white}{$\overline{\overline{\underline{\sum}}}$}(iii)\textcolor{white}{$\overline{\overline{\underline{\sum}}}$} & $\langle 1,\frac{1}{2}\Tr_{V_2},\frac{1}{4}\Tr_{C_2^2},\frac{1}{8}\Tr_{D_8}\rangle$ &  $\langle 1,\frac{1}{2}\Tr_{V_4},\frac{1}{8}\Tr_{G}\rangle$\\
\hline
\textcolor{white}{$\overline{\overline{\underline{\sum}}}$}(iv)\textcolor{white}{$\overline{\overline{\underline{\sum}}}$} & $\langle 1,\frac{1}{2}\Tr_{V_2},\frac{1}{4}\Tr_{C_4}\rangle$ &  $\langle 1,\frac{1}{2}\Tr_{V_4},\frac{1}{4}\Tr_{G}\rangle$\\
\hline
\textcolor{white}{$\overline{\overline{\underline{\sum}}}$}(v)\textcolor{white}{$\overline{\overline{\underline{\sum}}}$} & $\langle 1,\frac{1}{2}\Tr_{V_2},\frac{1}{4}\Tr_{V_4}\rangle$ &  $\langle 1,\frac{1}{4}\Tr_{V_4}\rangle$\\
\hline
\textcolor{white}{$\overline{\overline{\underline{\sum}}}$}(vi)\textcolor{white}{$\overline{\overline{\underline{\sum}}}$} & $\langle 1,\frac{1}{2}\Tr_{V_2},\frac{1}{4}\Tr_{C_2^2}\rangle$ &  $\langle 1,\frac{1}{2}\Tr_{V_4},\frac{1}{4}\Tr_{G}\rangle$\\
\hline
\textcolor{white}{$\overline{\overline{\underline{\sum}}}$}(vii)\textcolor{white}{$\overline{\overline{\underline{\sum}}}$} & $\langle 1,\frac{1}{2}\Tr_{W_2},\frac{1}{4}\Tr_{C_2^2}\rangle$ &  $\langle 1,\frac{1}{4}\Tr_{G}\rangle$\\
\hline
\textcolor{white}{$\overline{\overline{\underline{\sum}}}$}(viii)\textcolor{white}{$\overline{\overline{\underline{\sum}}}$} & $\langle 1,\frac{1}{2}\Tr_{C_2^2}\rangle$ &  $\langle 1,\frac{1}{2}\Tr_{G}\rangle$\\
\hline
\textcolor{white}{$\overline{\overline{\underline{\sum}}}$}(ix)\textcolor{white}{$\overline{\overline{\underline{\sum}}}$} & $\langle 1,\frac{1}{2}\Tr_{V_2}\rangle$ &  $\langle 1,\frac{1}{2}\Tr_{V_4}\rangle$\\
\hline
\textcolor{white}{$\overline{\overline{\underline{\sum}}}$}(x)\textcolor{white}{$\overline{\overline{\underline{\sum}}}$} & $\langle 1,\frac{1}{2}\Tr_{W_2}\rangle$ &  $\langle 1,\frac{1}{2}\Tr_{G}\rangle$\\
\hline
\end{tabular}
\end{table}
 
Hofmann and Johnston \cite[\S8.5]{MR4136552} described the implementation of an algorithm in 
\textsc{Magma} that, given a finite group $\Gamma$, a rational prime $p$, and $\Z[\Gamma]$-lattices 
$X$ and $Y$ contained in $\Q[\Gamma]$, determines whether the localisations $X_{p}$ and $Y_{p}$ are isomorphic over $\Z_{(p)}[\Gamma]$. Note that by \cite[Proposition (30.17)]{MR632548}, this is equivalent to checking whether the $p$-adic completions are isomorphic over $\Z_{p}[\Gamma]$.
In present situation, we are interested in understanding whether
$\Ind_D^G\mathfrak{A}_{K_\mathfrak{P}/\Q_2}$ is free over its associated order $\mathfrak{A}_{K/\Q,2}$
for each of the cases in Table \ref{tables4}. 
By Lemma~\ref{lemma:biggerorder} this condition is equivalent to $\Ind_D^G\mathfrak{A}_{K_\mathfrak{P}/\Q_2}$ being isomorphic to $\mathfrak{A}_{K/\Q,2}$ as $\Z_2[G]$-lattices.
Since $\Z_2[G]+\frac{1}{2^{n_1}}\Z_2[G]\Tr_{H_1}+\cdots+\frac{1}{2^{n_k}}\Z_2[G]\Tr_{H_k}$ is the completion of $\Z[G]+\frac{1}{2^{n_1}}\Z[G]\Tr_{H_1}+\cdots+\frac{1}{2^{n_k}}\Z[G]\Tr_{H_k}$, 
a computation using the aforementioned algorithm shows that 
$\Ind_D^G\mathfrak{A}_{K_\mathfrak{P}/\Q_2}$ is not free over $\mathfrak{A}_{K/\Q,2}$ 
in each of the ten cases in Table \ref{tables4} (for the implementation see \S\ref{codeS4}).
\end{proof}

\begin{remark}
Cases (v) and (ix) from Table \ref{tables4} can be tackled using Theorem \ref{thm:indringstar}(v) and cases (vii) and (x) using Theorem \ref{thm:indringstar}(v) combined with Proposition \ref{prop:projproj}. More precisely:
\begin{itemize}
\item for (v) we apply Theorem \ref{thm:indringstar}(v) with $H=V_4$ and $G=S_4$;
\item for (vii) we apply Theorem \ref{thm:indringstar}(v) with $H=C_2^2$ and $G=D_8$ and Proposition \ref{prop:projproj} inducing from $D_8$ to $\Gal(K/\Q)\cong S_4$;
\item for (ix) we apply Theorem \ref{thm:indringstar}(v) with $H=V_4$ and $G=S_4$;
\item for (x) we apply Theorem \ref{thm:indringstar}(v) with $H=W_2$ and $G=D_8$ and Proposition \ref{prop:projproj} inducing from $D_8$ to $\Gal(K/\Q)\cong S_4$.
\end{itemize}
Note that here $G$ is not necessarily the Galois group and $H$ is not necessarily one of the decomposition groups.
\end{remark}

\begin{remark}\label{rmk:generalproblemlattices}
The computations in the proof of Theorem \ref{thm:s4classification} 
show that each of the lattices considered is free over its associated order if and only if the lattice is a ring
if and only if the lattice is equal to its associated order. 
However, with the algorithm of \cite[\S 8.5]{MR4136552} we found that $\langle 1,\frac{1}{4}\Tr_{V_4},\frac{1}{8}\Tr_{D_8}\rangle$ is free over $\langle 1,\frac{1}{4}\Tr_{V_4},\frac{1}{8}\Tr_{G}\rangle$ (see \S\ref{codeS4} for the implementation).
\end{remark}

\subsection{Galois module structure of $A_5$-extensions of $\Q$}
In this subsection, we shall prove the following result, which is Theorem \ref{thm:a5intro} stated in the introduction.

\begin{theorem}\label{thm:A5main}
Let $K/\Q$ be a Galois extension with $\Gal(K/\Q) \cong A_{5}$. 
Then $\mathcal{O}_{K}$ is free over $\mathfrak{A}_{K/\Q}$ if and only if 
all three of the following conditions on $K/\Q$ hold:
\begin{enumerate}
\item $2$ is tamely ramified;
\item $3$ is tamely ramified or is weakly ramified with ramification index $6$; and
\item $5$ is tamely ramified or is weakly ramified with ramification index $10$.
\end{enumerate}
\end{theorem}

\begin{proof}[Proof of Theorem \ref{thm:A5main}]
By Corollary \ref{cor:A4S5A5-locally-free-implies-free}, 
$\mathcal{O}_{K}$ is free over $\mathfrak{A}_{K/\Q}$ if and only if 
$\mathcal{O}_{K,p}$ is free over $\mathfrak{A}_{K/\Q,p}$ for every rational prime $p$.
If $p$ is tamely ramified in $K/\Q$ then $\mathcal{O}_{K,p}$ is indeed free over $\mathfrak{A}_{K/\Q,p}$
by Theorem \ref{thm:noetherlocallyfree}. 
Thus it remains to consider the situation in which at least one of the primes $p=2,3,5$ is wildly ramified in 
$K/\Q$. 

We have the following lattice of the subgroups of $A_{5}$ up to conjugacy
(see, for instance, the GroupNames database \cite{groupnames}).
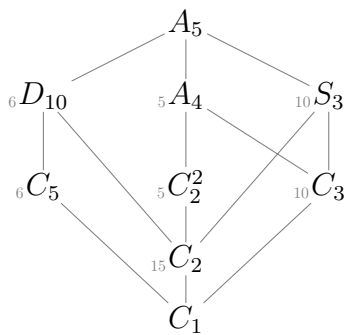
\[
\begin{tikzpicture}[scale=1.0,sgplattice]
  \node[char] at (2,0) (1) {\gn{C1}{C_1}};
  \node at (2,0.803) (2) {\gn{C2}{C_2}};
  \node at (3.88,1.76) (3) {\gn{C3}{C_3}};
  \node at (0.125,1.76) (4) {\gn{C5}{C_5}};
  \node at (2,1.76) (5) {\gn{C2^2}{C_2^2}};
  \node at (3.88,2.97) (6) {\gn{S3}{S_3}};
  \node at (0.125,2.97) (7) {\gn{D_{10}}{D_{10}}};
  \node at (2,2.97) (8) {\gn{A4}{A_4}};
  \node[char] at (2,3.93) (9) {\gn{A5}{A_5}};
  \draw[lin] (1)--(2) (1)--(3) (1)--(4) (2)--(5) (2)--(6) (3)--(6) (2)--(7)
     (4)--(7) (3)--(8) (5)--(8) (6)--(9) (7)--(9) (8)--(9);
  \node[cnj=2] {15};
  \node[cnj=3] {10};
  \node[cnj=4] {6};
  \node[cnj=5] {5};
  \node[cnj=6] {10};
  \node[cnj=7] {6};
  \node[cnj=8] {5};
\end{tikzpicture}
\] 
Here the subscript on the left denotes the number of conjugate subgroups. Recall that $A_{5}$ is simple and note that the subgroup lattice shows that isomorphic subgroups must be conjugate. 
Moreover, since $A_5$ is not soluble, no prime can have full decomposition group. 
We fix an isomorphism $G:=\Gal(K/\Q) \cong A_{5}$ and denote by
$A_4$, $D_{10}$ etc.\
a choice of subgroups of $G$ in 
such a way that whenever there is a containment between choices of conjugates of two such subgroups,
one of the subgroups is in fact contained in the other.

Suppose that $p=2$ is wildly ramified in $K/\Q$. 
Let $\mathfrak{P}$ be a prime of $K$ above $2$ and let $D(2)$ be its decomposition group. 
Then $D(2)$ must be isomorphic to $A_{4}$, $C_{2}^{2}$ or $C_{2}$. 
Hence in each of these cases $\mathcal{O}_{K_\mathfrak{P}}$ is free over $\mathfrak{A}_{K_\mathfrak{P}/\Q_2}$ (if $D(2)=A_4$, this follows from Lemma~\ref{lem:unique-A4-ext-of-Q2} and Theorem \ref{thm:weak}; 
otherwise this follows from Theorem \ref{thm:lettl-abs-abelian-p-adic}). 
By Proposition~\ref{prop:induction-in-free-rank-1-case}, $\textup{Ind}_{D(2)}^{G}\mathfrak{A}_{K_\mathfrak{P}/\Q_2} 
\cong \textup{Ind}_{D(2)}^{G} \mathcal{O}_{K_{\mathfrak{P}}} \cong \mathcal{O}_{K,2}$
as $\mathfrak{A}_{K/\Q,2}$-lattices. 
Thus we need to analyse when $\Ind_{D(2)}^G\mathfrak{A}_{K_\mathfrak{P}/\Q_2}$ is free over $\mathfrak{A}_{K/\Q,2}$. 
In each of the aforementioned possibilities for $D(2)$, we already know the structure of
$\mathfrak{A}_{K_\mathfrak{P}/\Q_2}$ from the proof of Theorem \ref{thm:a4}. 
Using Theorem \ref{thm:indringstar}(i)\&(ii) we can write all the possibilities for $\Ind_{D(2)}^G\mathfrak{A}_{K_\mathfrak{P}/\Q_2}$ and $\mathfrak{A}_{K/\Q,2}$.
We use the following notation: $\langle 1,\frac{1}{2}\Tr_{C_2}\rangle=\Z_2[G]+\frac{1}{2}\Z_2[G]\Tr_{C_2}$ etc.\ The results are shown in Table \ref{tablea5prime2} below.

\begin{table}[htb]
\centering
\caption{}\label{tablea5prime2}
\begin{tabular}{|c|c|c|}
\hline
\textcolor{white}{$\overline{\overline{\underline{\sum}}}$}& $\Ind_{D(2)}^{G}
\mathfrak{A}_{K_\mathfrak{P}/\Q_2}$ & $\mathfrak{A}_{K/\Q,2}$   \\
\hline
\textcolor{white}{$\overline{\overline{\underline{\sum}}}$}(i)\textcolor{white}{$\overline{\overline{\underline{\sum}}}$} & $\langle 1,\frac{1}{2}\Tr_{C_2}\rangle$ &  $\langle 1,\frac{1}{2}\Tr_{G}\rangle$\\
\hline
\textcolor{white}{$\overline{\overline{\underline{\sum}}}$}(ii)\textcolor{white}{$\overline{\overline{\underline{\sum}}}$} & $\langle 1,\frac{1}{2}\Tr_{C_2},\frac{1}{4}\Tr_{C_2^2}\rangle$ &  $\langle 1,\frac{1}{4}\Tr_{G}\rangle$\\
\hline
\textcolor{white}{$\overline{\overline{\underline{\sum}}}$}(iii)\textcolor{white}{$\overline{\overline{\underline{\sum}}}$} & $\langle 1,\frac{1}{2}\Tr_{C_2^2}\rangle$ &  $\langle 1,\frac{1}{2}\Tr_{G}\rangle$\\
\hline
\end{tabular}
\end{table}

As in the proof of Theorem \ref{thm:s4classification}, we can use the \textsc{Magma} implementation of the algorithm described in \cite[\S8.5]{MR4136552}. We can hence verify that in none of the above cases $\textup{Ind}_{D(2)}^{G}\mathfrak{A}_{K_\mathfrak{P}/\Q_2}\cong\mathcal{O}_{K,2}$ is free over $\mathfrak{A}_{K/\Q,2}$ (see \S\ref{codeA5} for the implementation). 

Now suppose that $p=3$ or $5$ and that $p$ is wildly ramified in $K/\Q$.
Let $\mathfrak{P}$ be a choice of a prime of $K$ above $p$ and let $D(p)$ be its decomposition group. 
There is no Galois extension $L/\Q_{3}$ such that $\Gal(L/\Q_{3}) \cong A_{4}$
(one can check this using the database of $p$-adic fields \cite{MR2194887}, for example).
Hence $D(p)$ must be isomorphic to either $D_{2p}$ or $C_{p}$,
which implies that $\mathcal{O}_{K_\mathfrak{P}}$ is free over $\mathfrak{A}_{K_\mathfrak{P}/\Q_p}$
by Theorems \ref{thm:lettl-abs-abelian-p-adic} and \ref{thm:padicberge}.
Thus by Proposition \ref{prop:induction-in-free-rank-1-case} we have that 
$\textup{Ind}_{D(p)}^{G}\mathfrak{A}_{K_\mathfrak{P}/\Q_p} 
\cong \textup{Ind}_{D}^{G} \mathcal{O}_{K_{\mathfrak{P}}} \cong \mathcal{O}_{K,p}$
as $\mathfrak{A}_{K/\Q,p}$-lattices, so that our goal is to analyse when $\textup{Ind}_{D(p)}^{G}\mathfrak{A}_{K_\mathfrak{P}/\Q_p}$ is free over $\mathfrak{A}_{K/\Q,p}$. 
We use the following notation: $\langle 1,\frac{1}{p}\Tr_{C_p}\rangle=\Z_p[G]+\frac{1}{p}\Z_p[G]\Tr_{C_p}$, etc.\ If $D(p)\cong C_p$ (in which case we can and do assume that $D(p)=C_{p}$), then as $K_\mathfrak{P}/\Q_p$ is wildly ramfied this implies that $\mathfrak{A}_{K_\mathfrak{P}/\Q_p}=\Z_p[D(p)]+\frac{1}{p}\Z_p[D(p)]\Tr_{D(p)}$, which is the unique maximal order in $\Q_{p}[D(p)]$. If $D(p)\cong D_{2p}$ (in which case we can and do assume that $D=D_{2p}$), we can use Theorem \ref{thm:locallyfreedihedral}: in case of almost-maximal ramification, $\mathfrak{A}_{K_\mathfrak{P}/\Q_p}=\Z_p[D(p)]+\frac{1}{p}\Z_p[D(p)]\Tr_{C_p}$ (which gives the same structure for $\textup{Ind}_{D(p)}^{G}\mathfrak{A}_{K_\mathfrak{P}/\Q_p} 
$ as when $D(p)=C_p$); otherwise, by Remark \ref{rmk:weakalmostmaximal}, $K_\mathfrak{P}/\Q_p$ is weakly and totally ramified and $\mathfrak{A}_{K_\mathfrak{P}/\Q_p}=\Z_p[D(p)]+\frac{1}{p}\Z_p[D(p)]\Tr_{D(p)}$. 
Hence there are two possibilities for $\Ind_{D(p)}^G\mathfrak{A}_{K_\mathfrak{P}/\Q_p}$ and $\mathfrak{A}_{K/\Q,p}$, shown in Table \ref{tablea5primep}.

\begin{table}[htb]
\centering
\caption{}\label{tablea5primep}
\begin{tabular}{|c|c|c|}
\hline
 \textcolor{white}{$\overline{\overline{\underline{\sum}}}$}& $\Ind_{D(p)}^G\mathfrak{A}_{K_\mathfrak{P}/\Q_p}$ & $\mathfrak{A}_{K/\Q,p}$   \\
\hline
\textcolor{white}{$\overline{\overline{\underline{\sum}}}$}(i)\textcolor{white}{$\overline{\overline{\underline{\sum}}}$} & $\langle 1,\frac{1}{p}\Tr_{C_p}\rangle$ &  $\langle 1,\frac{1}{p}\Tr_{G}\rangle$\\
\hline
\textcolor{white}{$\overline{\overline{\underline{\sum}}}$}(ii)\textcolor{white}{$\overline{\overline{\underline{\sum}}}$} & $\langle 1,\frac{1}{p}\Tr_{D_{2p}}\rangle$ &  $\langle 1,\frac{1}{p}\Tr_{G}\rangle$\\
\hline
\end{tabular}
\end{table}

We used the \textsc{Magma} implementation of the algorithm from \cite[\S8.5]{MR4136552} to verify that $\Ind_{D(p)}^{G}\mathfrak{A}_{K_\mathfrak{P}/\Q_p}$ is free over $\mathfrak{A}_{K/\Q,p}$ if and only if we are in case (ii), that is, precisely when $K_\mathfrak{P}/\Q_p$ is weakly ramified or, 
equivalently, when it is not almost maximally-ramified (see \S\ref{codeA5}).
\end{proof}

\begin{remark}
Note that from the proof of Theorem \ref{thm:a4}, we already knew that neither $\Z_2[A_4]+\frac{1}{2}\Z_2[A_4]\Tr_{C_2}$ nor $\Z_2[A_4]+\frac{1}{2}\Z_2[A_4]\Tr_{C_2}+\frac{1}{4}\Z_2[A_4]\Tr_{C_2^2}$ are even projective over their associated orders; induction from $A_4$ to $S_4$ and Proposition \ref{prop:projproj} permit us to conclude that $\Z_2[S_4]+\frac{1}{2}\Z_2[S_4]\Tr_{C_2}$ and $\Z_2[S_4]+\frac{1}{2}\Z_2[S_4]\Tr_{C_2}+\frac{1}{4}\Z_2[S_4]\Tr_{C_2^2}$ are not projective over their associated orders. Thus we can treat cases (i) and (ii) from Table \ref{tablea5prime2} without using the algorithm. 
\end{remark}

\begin{remark}\label{rmk:generalproblemlatticesa5}
Note that for $p=3$ and $p=5$ we found that $\langle 1,\frac{1}{p}\Tr_{D_{2p}}\rangle$ is free over $\langle 1,\frac{1}{p}\Tr_{G}\rangle$ without the two being equal. We also found with the algorithm from \cite{MR4136552} that $\langle 1,\frac{1}{2}\Tr_{A_4}\rangle$, which does not come from a ring of integers, is free over $\langle 1,\frac{1}{2}\Tr_G\rangle$ (see \S\ref{codeA5}).
\end{remark}

\appendix

\section{Computer calculations}\label{appendix}

\subsection{Determining freeness for $S_4$-extensions of $\Q_2$}\label{codeS4Q2}
Let $K/\Q$ be an $S_4$-extension with full decomposition group that is wildly ramified.
Here we describe how to use the \textsc{Magma} implementation of \cite[Algorithm 3.1(6)]{MR2422318} to check whether or not $\mathcal{O}_K$ is locally free at $2$ over $\mathfrak{A}_{K/\Q}$, or equivalently, whether $\mathcal{O}_{K_{\mathfrak{P}}}$ is free over $\mathfrak{A}_{K_\mathfrak{P}/\Q_2}$, where $\mathfrak{P}$ is the unique prime of $K$ above $2$. 
We used the database \cite{lmfdb} to find six number fields, each of which has a completion at $2$ equal 
to one of the six wildly ramified $S_4$-extensions of $\Q_2$ listed in the database of $p$-adic fields \cite{MR2194887}.
Note that two of these extensions of $\Q_{2}$ are weakly ramified, and so Corollary \ref{cor:inertianormal}(ii)
already shows that $\mathcal{O}_{K_{\mathfrak{P}}}$ is free over $\mathfrak{A}_{K_\mathfrak{P}/\Q_2}$ in both cases, but we include them anyway as an addition check.
The files \texttt{RelAlgKTheory.m} and \texttt{INB.m} referred to below are available on Werner Bley's website
\texttt{https://www.mathematik.uni-muenchen.de/$\sim$bley/pub.php}

We refer to the sample file \texttt{sample.m} from the article \cite{MR2422318}. 
Note that here we use the updated file \texttt{INB.m} from \cite{MR2813368} rather than the original file
\texttt{ao.m}.

\medskip

\begin{verbatim}
Attach("RelAlgKTheory.m");
Attach("INB.m");
P<x> := PolynomialRing(IntegerRing());
Polynomials := [ x^6 + x^4 + x^2 - 1,
x^6 - x^4 + 3*x^2 - 1,
x^6 + 3*x^4 + 11*x^2 + 11,
x^6 + 7*x^4 + 15*x^2 + 11,
x^6 - x^4 - 2*x^3 - x^2 + 1,
x^6 - 2*x^5 + 2*x^4 - 4*x^3 + 4*x^2 - 2*x + 2 ];
for i in [1..6] do
  L := NormalClosure(NumberField(Polynomials[i]));
  G, Aut, h := AutomorphismGroup(L);
  h := map<Domain(h)->Codomain(h) | g:->h(g^-1)>;
  OL := MaximalOrder(L);
  theta := NormalBasisElement(OL, h);
  Ath := ComputeAtheta(OL, h, theta);
  QG := GroupAlgebra(Rationals(), G);
  AssOrd := ModuleConductor(QG, Ath, Ath);
  rho := RegularRep(QG);
  M := ZGModuleInit(Ath`hnf, rho);
  isfree, w := IsLocallyFree(QG, AssOrd, M, 2);
  if isfree then
    print "we have local freeness at 2";
  else
    print "we do not have local freeness at 2";
  end if;
end for;
we do not have local freeness at 2
we do not have local freeness at 2
we do not have local freeness at 2
we do not have local freeness at 2
we have local freeness at 2
we have local freeness at 2
\end{verbatim}

\subsection{Determining local freeness at $2$ for $S_4$-extensions of $\Q$}\label{codeS4}
Here we describe how to use to use the \textsc{Magma} implementation of \cite[\S 8.5]{MR4136552}
to show that for an $S_{4}$-extension 
$K/\Q$ we have that $\mathcal{O}_K$ is not locally free at $2$ over $\mathfrak{A}_{K/\Q}$ if 
$K/\Q$ does not satisfy any of the conditions (i)--(iv) of Theorem \ref{thm:s4classification} and $\mathcal{O}_{K_\mathfrak{P}}$ is free over $\mathfrak{A}_{K_\mathfrak{P}/\Q_2}$, where $\mathfrak{P}$ is a prime of $K$ above $2$ (see Table \ref{tables4}). 
Moreover, we also prove the freeness claim of Remark \ref{rmk:generalproblemlattices}. 
The files \texttt{Iso.m}, \texttt{Lattices.m} and \texttt{Iso.spec} referred to below are contained in \texttt{Iso.zip}, available in the link to \cite{MR4136552} on Tommy Hofmann's website \texttt{https://www.thofma.com}

\medskip

\begin{verbatim}
AttachSpec("Iso.spec");
G := Sym(4);
W2 := sub<G | G!(1, 3)>;
V2 := sub<G | G!(1, 3)(2, 4)>;  
C4 := sub<G | G!(1, 2, 3, 4)>; 
V4 := sub<G | G!(1, 3)(2, 4),(1, 2)(3, 4)>; 
C22 := sub<G | G!(1, 3),(2, 4)>;  
D8 := sub<G | G!(1, 2, 3, 4),(1, 3)>;       
QG := GroupAlgebra(Rationals(), G);
trW2 := &+[ QG!h : h in W2];
trV2 := &+[ QG!h : h in V2];
trC4 := &+[ QG!h : h in C4];
trV4 := &+[ QG!h : h in V4];
trC22 := &+[ QG!h : h in C22];
trD8 := &+[ QG!h : h in D8];
trG := &+[ QG!h : h in G]; 
ZG := Order(Integers(), Basis(QG));
M1 := rideal< ZG | 1, trV2/2, trC4/4,trD8/8>;
A1 := rideal< ZG | 1, trV4/2, trG/8>;        
IsLocallyIsomorphic(QG, BasisMatrix(M1), BasisMatrix(A1), 2); 
false
M2 := rideal< ZG | 1, trV2/2, trV4/4,trD8/8>;
A2 := rideal< ZG | 1, trV4/4, trG/8>; 
IsLocallyIsomorphic(QG, BasisMatrix(M2), BasisMatrix(A2), 2);
false
M3 := rideal< ZG | 1, trV2/2, trC22/4,trD8/8>;                                                                     
IsLocallyIsomorphic(QG, BasisMatrix(M3), BasisMatrix(A1), 2);
false
M4 := rideal< ZG | 1, trV2/2, trC4/4>;
A4 := rideal< ZG | 1, trV4/2, trG/4>;
IsLocallyIsomorphic(QG, BasisMatrix(M4), BasisMatrix(A4), 2);
false
M5 := rideal< ZG | 1, trV2/2, trV4/4>;                       
A5 := rideal< ZG | 1, trV4/4>;        
IsLocallyIsomorphic(QG, BasisMatrix(M5), BasisMatrix(A5), 2);
false
M6 := rideal< ZG | 1, trV2/2, trC22/4>;                                                                            
IsLocallyIsomorphic(QG, BasisMatrix(M6), BasisMatrix(A4), 2);
false
M7 := rideal< ZG | 1, trW2/2, trC22/4>;                      
A7 := rideal< ZG | 1, trG/4>;                                
IsLocallyIsomorphic(QG, BasisMatrix(M7), BasisMatrix(A7), 2);
false
M8 := rideal< ZG | 1, trC22/2>;                              
A8 := rideal< ZG | 1, trG/2>;  
IsLocallyIsomorphic(QG, BasisMatrix(M8), BasisMatrix(A8), 2);
false
M9 := rideal< ZG | 1, trV2/2>;                               
A9 := rideal< ZG | 1, trV4/2>;
IsLocallyIsomorphic(QG, BasisMatrix(M9), BasisMatrix(A9), 2);
false
M10 := rideal< ZG | 1, trW2/2>;                                                            
IsLocallyIsomorphic(QG, BasisMatrix(M10), BasisMatrix(A8), 2);
false
M11 := rideal< ZG | 1, trV4/4, trD8/8>;                        
A11 := rideal< ZG | 1, trV4/4, trG/8>; 
IsLocallyIsomorphic(QG, BasisMatrix(M11), BasisMatrix(A11), 2);
true -31/4*Id(G) + (1, 4, 3, 2) + 5/4*(1, 3)(2, 4) - 5*(2, 3)
+ 5/4*(1, 2, 4) + 1/4*(1, 4, 3)+ (1, 3, 4, 2) + (2, 4, 3)+ (1, 4, 2, 3)
+ (1, 2, 3) + 5/4*(2, 3, 4) + 1/4*(1, 3, 2) + (2, 4) + 5/4*(1, 2)(3, 4)
+ 1/4*(1, 4)(2, 3)
\end{verbatim}

\subsection{Determing local freeness for $A_5$-extensions of $\Q$}\label{codeA5}
Here we describe how to use the \textsc{Magma} implementation of \cite[\S 8.5]{MR4136552}
to check local freeness in $A_5$-extensions of $\Q$ at the wildly ramified primes (see Tables \ref{tablea5prime2} and \ref{tablea5primep}). 
Moreover, we also prove the second freeness claim of Remark \ref{rmk:generalproblemlatticesa5}. 
The files \texttt{Iso.m}, \texttt{Lattices.m} and \texttt{Iso.spec} referred to below are contained in \texttt{Iso.zip}, available in the link to \cite{MR4136552} on Tommy Hofmann's website \texttt{https://www.thofma.com}

When \texttt{IsLocallyIsomorphic(QG, BasisMatrix(M), BasisMatrix(A), 2)} is `true', 
we suppress the full output, which includes an element $x\in\Q[G]$ such that $x(\Z_2\otimes_{\Z}M)=\Z_2\otimes_{\Z} A$ (whose existence is in our case equivalent to $\Z_2\otimes_{\Z}M$ being free over $\Z_2\otimes_{\Z} A$).

\medskip

\begin{verbatim}
AttachSpec("Iso.spec");
G:=Alt(5);
C2 := sub<G | G!(1, 2)(3, 4)>; 
C22 := sub<G | G!(1, 2)(3, 4),(1, 3)(2, 4)>; 
C3 := sub<G | G!(1, 2, 3)>;                 
D6 := sub<G | G!(1, 2)(4, 5),(1, 2, 3)>; 
C5 := sub<G | G!(1, 2, 3, 4, 5)>;       
D10 := sub<G | G!(2 ,5)(3, 4),(1, 2, 3, 4, 5)>;
Alt4 := sub<G | G!(1 ,2)(3, 4),(1, 2, 3)>; 
QG := GroupAlgebra(Rationals(), G);
trC2 := &+[ QG!h : h in C2];      
trC22 := &+[ QG!h : h in C22];
trC3 := &+[ QG!h : h in C3];  
trD6 := &+[ QG!h : h in D6];
trC5 := &+[ QG!h : h in C5];
trD10 := &+[ QG!h : h in D10];                      
trAlt4 := &+[ QG!h : h in Alt4]; 
trG := &+[ QG!h : h in G];  
ZG := Order(Integers(), Basis(QG));
M1 := rideal< ZG | 1, trC2/2>;     
A1 := rideal< ZG | 1, trG/2>; 
IsLocallyIsomorphic(QG, BasisMatrix(M1), BasisMatrix(A1), 2); 
false
M2 := rideal< ZG | 1, trC2/2, trC22/4>;                      
A2 := rideal< ZG | 1, trG/4>;                                
IsLocallyIsomorphic(QG, BasisMatrix(M2), BasisMatrix(A2), 2); 
false
M3 := rideal< ZG | 1, trC22/2>;                                                                                    
IsLocallyIsomorphic(QG, BasisMatrix(M3), BasisMatrix(A1), 2); 
false
M4 := rideal< ZG | 1, trC3/3>;                               
A4 := rideal< ZG | 1, trG/3>;                                
IsLocallyIsomorphic(QG, BasisMatrix(M4), BasisMatrix(A4), 3); 
false
M5 := rideal< ZG | 1, trD6/3>;                               
IsLocallyIsomorphic(QG, BasisMatrix(M5), BasisMatrix(A4), 3); 
true
M6 := rideal< ZG | 1, trC5/5>;                               
A6 := rideal< ZG | 1, trG/5>;                                
IsLocallyIsomorphic(QG, BasisMatrix(M6), BasisMatrix(A6), 5); 
false
M7 := rideal< ZG | 1, trD10/5>;                              
IsLocallyIsomorphic(QG, BasisMatrix(M7), BasisMatrix(A6), 5); 
true                         
M8 := rideal< ZG | 1, trAlt4/2>;                               
IsLocallyIsomorphic(QG, BasisMatrix(M8), BasisMatrix(A1), 2); 
true 
\end{verbatim}

\bibliography{Associated-orders-bib.bib}{}
\bibliographystyle{amsalpha}

\end{document}